\documentclass[a4paper]{article}
\usepackage{a4, ams, amsmath, amssymb, graphics, subfigure, fullpage, color}% 
\usepackage[pdftex]{graphicx}%remove if elsarticle
\usepackage{fancyhdr}
\usepackage{multirow}
\usepackage{amsthm}
\usepackage[colorlinks]{hyperref}
\hypersetup{colorlinks,citecolor=green,filecolor=black,linkcolor=blue,urlcolor=blue}

\newtheorem{definition}{Definition}
\newtheorem{corollary}{Corollary}
\newtheorem{proposition}{Proposition}
\newtheorem{lemma}{Lemma}
\newtheorem{example}{Example}
\newtheorem{remark}{Remark}
\newtheorem{algorithm}{Algorithm}
\usepackage{upgreek} % for capital greeks in bib
\newcommand{\MyBib}{./MyBib}

\usepackage{lscape}
\newcommand{\InclTDC}

\numberwithin{equation}{section}
\numberwithin{theorem}{section}
\numberwithin{corollary}{section}
\numberwithin{definition}{section}
\numberwithin{lemma}{section}
\numberwithin{remark}{section}
\numberwithin{example}{section}

\newcommand{\cF}{\mathcal{F}}

\newcommand{\cE}{\mathcal{E}}
\newcommand{\FF}{\mathbb{F}}
\renewcommand{\R}{\mathbf{R}}

\DeclareMathOperator{\E}{\mathbb{E}}

\DeclareMathOperator{\e}{\textrm{e}}

\DeclareMathOperator{\sign}{\textrm{sign}}

\newcommand{\beq}{\begin{equation}}
\newcommand{\eeq}{\end{equation}}
\newcommand{\beqn}{\begin{eqnarray}}
\newcommand{\eeqn}{\end{eqnarray}}
\newcommand{\bfig}{\begin{figure}}
\newcommand{\efig}{\end{figure}}
\newcommand{\btab}{\begin{table}}
\newcommand{\etab}{\end{table}}

\title{Piecewise Constant Martingales and Lazy Clocks}
\author{Christophe PROFETA\footnote{Laboratoire de Math\'ematiques et de Mod\'elisation d'\'Evry (LaMME), Université d'\'Evry, CNRS. Boulevard de France 23, 91025, Evry , France. E-mail: \texttt{christophe.profeta@univ-evry.fr}.}~~\& Fr\'ed\'eric VRINS\footnote{Louvain Finance Center and CORE, Université catholique de Louvain. Voie du Roman Pays 34, 1348 Belgium. E-mail: \texttt{frederic.vrins@uclouvain.be}.}}
\date{First version: 18th October 2016. This version: \today}

\begin{document}
\maketitle
\begin{abstract} This paper discusses the possibility to find and construct \textit{piecewise constant martingales}, that is, martingales with piecewise constant sample paths evolving in a connected subset of $\mathbb{R}$. After a brief review of standard possible techniques, we propose a construction based on the sampling of latent martingales $\tilde{Z}$ with \textit{lazy clocks} $\theta$. These $\theta$ are time-change processes staying in arrears of the true time but that can synchronize at random times to the real clock. This specific choice makes the resulting time-changed process  $Z_t=\tilde{Z}_{\theta_t}$ a martingale (called a \textit{lazy martingale}) without any assumptions on $\tilde{Z}$, and in most cases, the lazy clock $\theta$ is adapted to the filtration of the lazy martingale $Z$. This would not be the case if the stochastic clock $\theta$ could be ahead of the real clock, as typically the case using standard time-change processes. The proposed approach yields an easy way to construct analytically tractable lazy martingales evolving on (intervals of) $\mathbb{R}$.
\end{abstract}

\bigskip

The authors are grateful to M. Jeanblanc, D. Brigo and K. Yano for stimulating discussions about an earlier version of this manuscript. This research benefited from the support of the ``\textit{Chaire March\'es en Mutation}'', F\'ed\'eration Bancaire Fran\c{c}aise.

%\newpage

% ====================================================================================
\section{Introduction}
% ====================================================================================

In the literature, \textit{pure jump processes} defined on a filtered probability space $(\Omega,\cF,\FF,\Pr)$, where $\FF:=(\cF_t,0\leq t\leq T)$ and $\cF:=\cF_T$, are often referred to as stochastic processes having no diffusion part. In this paper we are interested in a subclass of pure jump (PJ) processes: \textit{piecewise constant (PWC) martingales} defined as follows. 
\begin{definition}[Piecewise constant martingale]
A piecewise constant $\FF$-martingale $Z$ is a c\`adl\`ag $\FF$-martingale whose jumps $\Delta Z_s=Z_s-Z_{s^-}$  are summable (i.e. $\sum_{s\leq T} |\Delta Z_s|<+\infty$ a.s.) and such that for every $t\in [0,T]$ :
$$Z_t =Z_0 + \sum_{s\leq t} \Delta Z_s\;.$$
In particular, the sample paths $Z(\omega)$ for $\omega \in \Omega$ belong to the class of piecewise constant functions of time.
\end{definition}
Note that an immediate consequence of this definition is that a PWC martingale has finite variation.
%In fact, PWC martingales form the set of compensated pure jump processes whose compensators are piecewise constant. 
Such type of processes may be used to represent martingales observed under partial (punctual) information, e.g. at some (random) times. One possible field of application is mathematical finance, where discounted price processes are martingales under an equivalent measure. Without additional information, a reasonable approach may consist in assuming that discounted prices remain constant between arrivals of market quotes, and jump to the level given by the new quote when a new trade is done. More generally, this could represent conditional expectation processes (i.e. ``best guess'') where information arrives in a random and discontinuous way.  An interesting application in that respect is the modeling of \textit{quoted recovery rates}. They correspond to the market's view of a firm's recovery rate $R$ upon default. Being conditional expectations of random variables in $[0,1]$ associated to remote events, they are martingales evolving in the unit interval, whose trajectories remain constant for long period of times, but jumps from time to time, when dealers update their views to specialized data providers.\medskip

Pure jump martingales can easily be obtained by taking the difference of a pure jump increasing process with a predictable, grounded, right-continuous process of bounded variation (called \textit{compensator}). The simplest example is probably the compensated Poisson process of parameter $\lambda$ defined by $(M_t=N_t-\lambda t, \,t\geq0)$. This process is a pure jump martingale with piecewise \textit{linear} sample paths, hence is not a PWC martingale as $\sum_{s\leq t} \Delta M_s = N_t \neq M_t$.  
%Typical examples are compensated Poisson processes or compensated Compound Poisson processes whose jumps size is exponentially distributed. They are pure jump martingales but with piecewise \textit{linear} sample paths. 
Clearly, not all martingales having no diffusion term are piecewise linear. For example, the Az\'ema martingale $M$ defined as 
\beqn
M_t&:=&\E[W_t|\sigma(\sign(W_s))_{s\leq t}]=\sign(W_t)\sqrt{\frac{\pi}{2}}\sqrt{t-g_t^0(W)}\;,\nonumber\\
g_t^0(W)&:=&\sup \{s\leq t, W_s=0\}\label{eq:azema}
\eeqn
where $W$ is a Brownian motion, is essentially piecewise square-root (see e.g. Section 8 of~\cite{Prott05} for a detailed analysis of this process). 
%%
%\begin{remark}
%It is possible to turn any pure jump martingale $M$ into a piecewise linear martingale $\bar{M}$ by defining $\bar{M}_t$ via $d\bar{M}_t:=d[M,M]_t-d\E[(M_t-M_{t^-})^2|\cF^M_{t^-}]$. In the case of the above piecewise square-root martingale $M$ for example, $\bar{M}$, $\bar{M}_t:=[M,M]_t-\frac{\pi}{4}t$ is a piecewise linear martingale in its own filtration. The size of the $i$-th jump of $\bar{M}$ is deterministically known to be $M_{\tau_i-}^2$ since at the $i$-th jump time $\tau_i$ (i.e. the $i$-th hitting time of $W$ to 0) the process $M$ jumps from its current value to zero. Hence, it is not a Compound Poisson process in spite of the fact that its compensator is linear.
%\end{remark}
%%
Similarly, the Geometric Poisson Process $\e^{N_t\log(1+\sigma)-\lambda\sigma t}$ is a positive martingale with piecewise negative exponential sample paths~\cite[Ex 11.5.2]{Shrev04}. \medskip

%Hence, it is possible to find pure jump martingales whose compensator trajectories are not linear. 
In Section 2, we present several routes to construct PWC martingales. We then introduce a different approach in Section 3, adopting a time-changed technique. This method proves to be very flexible as the time-changed and the latent processes have the same range (if not time-dependent).

% ====================================================================================
\section{Piecewise constant martingales}
% ====================================================================================

Most of the ``usual'' martingales with no diffusion term fail to have piecewise constant sample paths. However, finding such type of processes is not difficult. We provide below three different methods to construct such type of processes. Yet, not all are equally powerful in terms of tractability. The last method proves to be quite appealing in that it yields PWC martingales whose range can be any connected set. % and whose marginals have known distribution (up to an integral).

% ------------------------------------------------------------------------------------
\subsection{An autoregressive construction scheme}
% ------------------------------------------------------------------------------------
We start by looking at a subset of PWC martingales, namely step martingales. These are martingales  whose paths belong to the space of step functions on any bounded interval. As a consequence, a step martingale $Z$ admits a finite number of jumps on $[0,T]$ taking places at, say $(\tau_k,\; k\geq 1)$, and may be decomposed as (with $\tau_0:=0$) 
$$Z_t = Z_0+\sum_{k=1}^{+\infty} (Z_{\tau_k} - Z_{\tau_{k-1}})1_{\{\tau_k \leq t\}}\;.$$

Looking at such decomposition, we see that step martingales may easily be constructed by an autoregressive scheme.

\begin{proposition}
Let $(M_n, n\in \mathbb{N})$ be a martingale such that $\sup_{i\geq1}\E[|M_i-M_{i-1}|]<+\infty.$ 
Let $(\tau_k,\; k\geq 1)$ be an increasing sequence of random times, independent from $M$, and set $A_t:=\sum_{k=1}^{+\infty}1_{\{\tau_k \leq t\}}$. We assume that $\E[A_t]<+\infty$. Then, the process
$$Z_t := M_0 +\sum_{k=1}^{+\infty} (M_k - M_{k-1})1_{\{\tau_k \leq t\}}= M_{A_t} $$
is a step martingale with respect to its natural filtration.
\end{proposition}

\begin{proof}
We first have
$$\E[|Z_t|] \leq \E[|M_0|]+\left(\sup_{i\geq1}\E[|M_i-M_{i-1}|]\right) \sum_{k=1}^{+\infty} \Pr(\tau_k\leq t) <+\infty$$
which proves that $Z_t$ is integrable. The martingale property is then an immediate consequence of the increasing time change $A$.
\end{proof}
%\textcolor{red}{bien vu ! si tu définis $Y_k:=6a/(\pi k)^2+6(b-a)/(\pi k)^2 U_k$, $U_k$ i.i.d. $\mathcal{U}(0,1)$ tu obtiens la CF comme une somme infinie de produits de CF d'uniformes. On a donc l'expression de la CF en raison de l'independance, mais ce n'est pas beau.}

\begin{example}\label{exa:sumcox}
Let $N$ be a Cox process with intensity $\lambda=(\lambda_t)_{t\geq 0}$ and $\tau_1,\ldots,\tau_{N_t}$ be the sequence of jump times of $N$ on $[0,t]$ with $\tau_0:=0$. If $(Y_k, k\geq 1)$ is a family of independent and centered random variables, then 
$$Z_t:=Z_0+\sum_{k=1}^\infty Y_{k}1_{\{\tau_k\leq t\}}=Z_0+\sum_{k=1}^{N_t} Y_{k}~,~~~Z_0\in\mathbb{R}$$
is a PWC martingale. Note that we may choose the range of such a PWC martingale by taking bounded random variables. For instance, if $Z_0=0$ and for any $k\geq 1$,
$$\Pr\left(\frac{6a}{\pi^2 k^2} \leq Y_k \leq \frac{6b}{\pi^2 k^2}  \right) =1$$
with $a<0<b$, then for any $t\geq0$, we have $Z_t\in[a,b]$ a.s.
\end{example}

%An easy way to build PWC martingales is to adopt a recursive scheme, adding zero-mean independent increments at some random times.
%
%\begin{theorem} Let $N$ be a Cox process with intensity $\lambda=(\lambda_t)_{t\geq 0}$ and $\tau_1,\ldots,\tau_{N_t}$ be the sequence of jump times of $N$ on $[0,t]$ with $\tau_0:=0$ and $Y_{i}$ are i.i.d random variables with $\E[Y_{i}]=0$. Then the process $Z$ defined as
%
%\beq
%Z_t:=Z_0+\sum_{i=1}^{N_t} Y_{i}=Z_0+\sum_{i=1}^\infty Y_{i}1_{\{\tau_i\leq t\}}~,~~~Z_0\in\mathbb{R}
%\eeq
%is a PWC martingale (with respect to its natural filtration).
%\end{theorem}
%
%\begin{proof} It is obvious that $Z$ is constant between the jump times of $N$ and jumps at $\tau_i, i>1$ . It is thus a pure jump process. Define $\tau(s):=\tau_{N_s}$ the last jump time of $N$ prior to $s$. Hence, $Z_{\tau(s)}=Z_s$ and
%
% \beqn
%\E[Z_t|\cF_s]&=&Z_s+\E[Z_t-Z_s|\cF_s]\\
%&=&Z_s+\E[Z_t-Z_s|Z_s]\\
%&=&Z_s+\E[Z_t-Z_s|Z_{\tau(s)}]\\
%&=&Z_s+\sum_{i:N_s< i\leq N_t} \E[Y_{i}|Z_{\tau(s)}]\\
%&=&Z_s\;.
%\eeqn
%\end{proof}
%
%Depending on the distribution of $Y_i\sim Y$, the conditional distribution of $Z_t(N_t)$ can take a relatively simple form. If for instance $Y\sim\cN(0,\sigma)$ then $Z_t(n)\sim\cN(0,\sigma\sqrt{n})$. If $Y_i\sim Y-1/\gamma$ where $Y\sim Exp(\gamma)$ then $Z_t\sim Y_t(n)-n/\gamma$ where $Y_t(n)\sim\Gamma(n,\gamma)$. \medskip
%

The above construction scheme provides us with a simple method to construct PWC martingales. Yet, it suffers from two restrictions. First, the distribution of $Z_t$ requires averaging the conditional distribution with respect to the Poisson distribution of rate $\lambda$, i.e. an infinite sum. Second, a control on the range of the resulting martingale requires  strong assumptions. 
%One might try to relax the i.i.d. assumption of the $Y_i$'s. 
In Example \ref{exa:sumcox}, the $Y_i$'s are independent but their support decreases as $1/k^2$. One might try to relax the independence assumption by drawing $Y_i$ from a distribution whose support is state dependent like $[a-Z_{\tau_{i-1}}, b-Z_{\tau_{i-1}}]$, in which case $Z_t\in [a,b]$ for all $t\in [0,T]$. By doing so however, we typically loose the tractability of the distribution. In Example \ref{exa:sumcox} for instance, the characteristic function can be found in closed form, but it features an infinite sum (over the Poisson states) of products (of increasing size) of characteristic functions associated to the random variables $(Y_i)$. In the sequel, we address these drawbacks by proposing another construction scheme, that would provide us with more tractable expressions.

%The second limitation can be circumvented by relaxing independence between the $Y_i$'s. If for instance $Y_i$ is drawn from a distribution whose support is state dependent like $[a-Z_{\tau_{i-1}}, b-Z_{\tau_{i-1}}]$, then $Z_t\in [a,b]$ for all $t\in [0,T]$. By relaxing the independence assumption however, we loose the distribution's tractability. In the sequel, we address these drawbacks by adopting another construction scheme.

% ------------------------------------------------------------------------------------
\subsection{PWC martingales from PJ martingales with vanishing compensator} 
% ------------------------------------------------------------------------------------

As hinted in the introduction, PWC martingales can be easily obtained by taking the difference of two pure jump processes whose compensators cancel out. We start by looking at subordinators. 
\begin{lemma}[Pure jump martingales constructed from subordinators]
Let $J^1$ and $J^2$ be two i.i.d. subordinators, with characteristic exponent :$$\varphi_{J_t^{1}}(u):=\E\left[e^{i u J_t^1 }\right] = \exp\left(- t \int_0^{+\infty} (1-e^{iu x}) \nu(dx)\right).$$
We assume that the L\'evy measure $\nu$ satisfies the integrability condition $\int_0^{+\infty} x \nu(dx)<+\infty$. Then,
$Z:=J^1-J^2$ is a PWC symmetric martingale whose characteristic function is given by $$\varphi_{Z_t}(u) = \exp\left(-2t \int_0^{+\infty} (1-\cos(u x)) \nu(dx)\right)\;.$$
\end{lemma}

\begin{proof}
Observe first that the assumption $\int_1^{+\infty} x\nu(dx)<+\infty$ implies that $J^1$ is integrable, while $\int_0^{1} x\nu(dx)<+\infty$ implies that  $J^1$  admits the decomposition  $J_t^1 = \sum_{s\leq t}\Delta J_s$, see \cite[p.15]{Bert96}. The result then follows from the fact that $M_t^1=  \sum_{s\leq t} \Delta J_s^1 - t \int_0^{+\infty} x \nu(dx)$ is a martingale.
\end{proof}

As obvious examples, one can mention the difference of two independent Gamma or Poisson processes of same parameters. Note that stable subordinators are not allowed here, as they do not fulfill the integrability condition. We give below the probability density of these two examples :

\begin{example}
Let $N^1,N^2$ be two independent Poisson processes with parameter $\lambda$. Then, $Z:=N^1-N^2$ is a step martingale taking integer values, with marginals given by the Skellam distribution with parameters $\mu_1=\mu_2=\lambda$ :
\beq
f_{Z_t}(k) = e^{-2\lambda t}I_{|k|}(2\lambda t),\qquad k\in \mathbb{Z}\;,
%F_{Z_t}(y)=\Pr(Y_t\leq y)=e^{-2\lambda t} \sum_{k=-\infty}^{\lfloor z \rfloor} I_{|k|}(2\lambda t)
\eeq
where $I_{k}$ is the modified Bessel function of the first kind.
\end{example}
\noindent

\begin{example}
Let $\gamma^1,\gamma^2$ be two independent Gamma processes with parameters $a,b>0$. Then, $Z:=\gamma^1-\gamma^2$ is a PWC martingale with marginals given by
\beq
f_{Z_t}(z)=   \frac{b}{\sqrt{\pi}\Gamma(at)} \left|\frac{bz}{2}\right|^{at-\frac{1}{2}}  K_{\frac{1}{2}-at}\left(b|z|\right)\;,
\eeq
where $K_\beta$ denotes the modified Bessel function of the second kind with parameter $\beta\in \R$.
\end{example}
\begin{proof}
The probability density of $Z_t$ is given, for $2at>1$, by the inverse Fourier transform, see \cite[p.349 Formula 3.385(9)]{GraRyz} :
$$
f_{Z_t}(z) = \frac{1}{2\pi} \int_\mathbb{R}  \frac{e^{-iu z}}{\left(1+i\frac{u}{b}\right)^{at}\left(1-i\frac{u}{b}\right)^{at}} du\;.
%&= 2^{-a t} \frac{|bz|^{at-1}}{\Gamma(at)} \sqrt{\frac{|bz|}{\pi}} K_{1/2-at}\left(2|bz|\right)
%=  \frac{1}{\sqrt{2\pi}\Gamma(at)} \left|\frac{bz}{2}\right|^{at-\frac{1}{2}}  K_{\frac{1}{2}-at}\left(2|bz|\right).
$$
The result then follows by analytic continuation.
\end{proof}

Note that more generally, a similar proof allows to characterize the centered L\'evy processes which are PWC martingales.
\begin{proposition}\label{prop:levy}
A centered L\'evy process $L$ is a PWC martingale if and only if it has no drift, no Brownian component and its L\'evy measure $\nu$ satisfies the integrability condition 
$\int_\mathbb{R} |x| \nu(dx)<\infty$, i.e. its L\'evy triple is $(0,0,\nu)$ with $\nu$ integrable as above.
\end{proposition}

%\begin{proof}
%Observe first that $L$ is martingale since it is centered and $\int_{\mathbb{R}\backslash(-1,1)} |x| \nu(dx)<\infty$ (which ensures the integrability). Then, the fact that $L$ has no drift, no Brownian component and that $\int_{-1}^1  |x| \nu(dx)<+\infty$ implies that $L$ admits the decomposition : 
%$$L_t = \sum_{s\leq t} \Delta L_s$$
%which is the definition of a PWC martingale.
%\end{proof}

We conclude this section with an example of PWC martingale which does not belong to the family of L\'evy processes but has the interesting feature to evolve in a time-dependent range. %\textcolor{red}{I'd rather use $g^0_t(W_i)$}
\begin{lemma}
Let $W^1, W^2$ be two independent Brownian motions. For $i=1,2$ set 
$$g^0_t(W^i) := \sup\{s\leq t,\; W^i_s=0\}.$$
Then, $Z:=g^0(W^1) - g^0(W^2)$ is a 1-self-similar PWC martingale which evolves in the cone $\{[-t, t], t\geq0\}$. Its Laplace transform admits the expansion :
$$\varphi_{Z_t}(iu)=\E\left[e^{-u Z_t}\right] = \sum_{k=0}^{+\infty}  \frac{(2k)!}{(k!)^4}\left(\frac{u t}{4}\right)^{2k}$$
and its cumulative distribution function (for $t>0$) is given, for $-t\leq z\leq t$, by :
$$F_{Z_t}(z)=  \frac{1}{2}+\emph{sign}(z) \frac{2}{\pi^2} \int_0^{\frac{\pi}{2}}\ln\left(\tan(x)\frac{|z|}{t} + \sqrt{1+\tan^2(x)\frac{z^2}{t^2}}\right)  \frac{dx}{\cos(x)}\;. $$
\end{lemma}

\begin{proof}
By Protter \cite[Theorem 87]{Prott05}, the processes $\left( g^0_t(W^i) - \frac{t}{2},\, t\geq0\right)$ are  martingales, hence so is $Z$. Denoting by $M$ the Az\'ema martingale (\ref{eq:azema}), the PWC property follows from the fact that the event $\{g_s^0(W)\neq g_{s^-}^0(W)\}$ implies $\{W_s=0\cap g_s^0(W)=s\}$ hence
\begin{align*}
 g_t^0(W) = \frac{2}{\pi} [M, M]_t &= \frac{2}{\pi}\sum_{s\leq t} (\Delta M_s)^2=\sum_{s\leq t} \left( - \text{sign}(W_{s^-})\sqrt{s - g_{s^-}^0(W)}\right)^2\\&=\sum_{s\leq t} g_s^0(W)-g_{s^-}^0(W)\;.
\end{align*}
Next, the self-similarity of $Z$ comes from that of $g^0(W)$, which further implies that for $t\geq0$ :
$$Z_t \sim t \,(g_1^0(W^1) - g_1^0(W^2))\quad \in [-t,t ]\;.$$
Finally, since $g_1^0(W)$ follows the Arcsine law, we deduce on the one hand, using a Cauchy product, that :  
\begin{align*}
\E\left[e^{-u Z_t}\right]  = I_0^2\left(\frac{u t}{2}\right) &=  \sum_{k=0}^{+\infty}  \left(\frac{u t}{4}\right)^{2k}  \sum_{i=0}^k \frac{1}{(i! (k-i)!)^2}\\
&=\sum_{k=0}^{+\infty}  \left(\frac{u t}{4}\right)^{2k}   \frac{1}{(k!)^2} \sum_{i=0}^k \binom{k}{i}^2=\sum_{k=0}^{+\infty}  \left(\frac{u t}{4}\right)^{2k}   \frac{1}{(k!)^2}  \binom{2k}{k}\;.
\end{align*}
On the other hand, the density of $Z_1$ is given by the convolution, for $z\in [0,1]$ :
$$
f_{Z_t}(z)= \frac{1}{\pi^2} \int_0^{1-z} \frac{1}{\sqrt{x(1-x)}} \frac{1}{\sqrt{(z+x)(1-z-x)}}  dx= \frac{2}{\pi^2} F\left(\frac{\pi}{2}, \sqrt{1-z^2}\right)\;,
$$
%\frac{2}{\pi^2} \int_0^{\frac{\pi}{2}} \frac{d\alpha }{\sqrt{1-(1-z^2)\sin^2(\alpha)}}  d\alpha
where $F$ denotes the incomplete elliptic integral of the first kind, see \cite[p.275, Formula 3.147(5)]{GraRyz}. This yields, by symmetry and scaling :
$$f_{Z_t}(z)= \frac{2}{\pi^2} \int_0^{\frac{\pi}{2}} \frac{dx}{\sqrt{t^2 \cos^2(x) +  z^2 \sin^2(x)}}\; 1_{\{0<|z|\leq t\}} $$
and the resulting cumulative distribution function is obtained upon integration in $z$.
\end{proof}

%The equality of parameters of $J^1,J^2$ is sufficient but not necessary: what matters is that the compensators of $J^1,J^2$ offset eachother (e.g. we could take two Compound Poisson processes provided that the product of jump size and jump rate coincide for both processes). This construction scheme is a convenient procedure to find PWC martingales in $\R$: if the Laplace transforms $\varphi^{J^i_t}(x)$ of the $J^i$'s are known, the Laplace transform $\varphi^{Z_t}(x)$ of $Z:=J^1-J^2$ is just $\varphi^{J^1_t}(x)\varphi^{J^2_t}(-x)$. If we take $J^i$ to be L\'evy subordinators with Laplace exponent $\psi^i$, this further simplifies to $\varphi^{Z_t}(x)=\e^{-t(\psi^1(x)+\psi^2(-x))}$.\medskip
Both the recursive and the vanishing compensators approaches are rather restrictive in terms of attainable range and analytical tractability. 
In the next section, we provide a more general method that can be used to build PWC martingales to any connected set of $\mathbb{R}$ (compatible with the martingale property, i.e. non-decreasing w.r.t. time) in a simple and tractable way.

% ------------------------------------------------------------------------------------
\subsection{PWC martingales using time-changed techniques}
% ------------------------------------------------------------------------------------

In this section, we construct a PWC martingale $Z$ by time-changing a latent ($\Pr,\FF$)-martingale $\tilde{Z}=(\tilde{Z}_t)_{t\in [0,T]}$ with the help of a suitable \textit{time-change process} $\theta$. 

\begin{definition}

A $\FF$ time-change process $\theta=(\theta_t)_{\in [0,T]}$ is a stochastic process satisfying
\begin{itemize}
\item $\theta_0=0$
\item for any $t\in [0,T]$, $\theta_t$ is $\mathcal{F}_t$-measurable (i.e. $\theta$ is adapted to the filtration $\FF$)
\item the map $t\mapsto \theta_t$ is c\`adl\`ag a.s. non-decreasing
\end{itemize}
\end{definition}

Under mild conditions stated below, $Z:=(\tilde{Z}_{\theta_t})_{t\geq 0}$ is proven to be a martingale on $[0,T]$ with respect to its own filtration, with the desired piecewise constant behavior. 
Most results regarding time-changed martingales deal with continuous martingales time-changed with a continuous process~\cite{jyc:3m,RevYor99}. This does not provide a satisfactory solution to our problem as the resulting martingale will obviously have continuous sample paths. On the other hand, it is obvious that not all time-changed martingales remain martingales, so that conditions are required on $Z$ and/or on $\theta$.
\begin{remark}
Every $\FF$-semi-martingale time-changed with a $\FF$-adapted process remains a semi-martingale but not necessarily a martingale. For instance, setting $\tilde{Z}=W$ and $\theta_t=\inf\{s:W_s>t\}$ then $\tilde{Z}_{\theta_t}=t$. Also, even if $\theta$ is independent from $\tilde{Z}$, $Z$ may fail to be a martingale in the above filtration because of integrability issues. For example if $\tilde{Z}=W$ and $\theta$ is an independent $\alpha$-stable subordinator with $\alpha=1/2$ then the time-changed process $Z$ is not integrable: $\E[|\tilde{Z}_{\theta_t}|~|\theta_t]=\sqrt{\frac{2}{\pi}}\sqrt{\theta_t}$ and $\E[\sqrt{\theta_t}]$ is undefined. 
\end{remark}

A sufficient condition to ensure that the time-changed martingale remains a martingale is to constraint $\tilde{Z}$ to be positive independent from $\theta$. Taking as $\theta$ a time-change process independent from $\tilde{Z}>0$, this result allows one to construct piecewise constant martingales having the same range as $\tilde{Z}$.
% in $[a,b]$, $[a,\infty)$ or $(-\infty,a]$ for any $a\in\mathbb{R}$. 
This is shown in the next lemma~\cite[Lemma 15.2 ]{Cont04}
\footnote{This result is derived in a chapter considering continuous time change processes (``this rules out subordinators''). However the authors do not rely on this assumption in the proof. Moreover, they use as a counter-example $\tilde{Z}=W$ a Brownian motion and $\theta$ a subordinator, suggesting that subordinators fit in the scope of stochastic clock processes considered in this lemma.}

\begin{lemma} Let $\tilde{Z}$ be a positive martingale (in its own filtration) and $\theta$ be an independent time-change process. Then, the time-changed process $Z$ is again a martingale in the filtration generated by the time-changed process $\tilde{Z}$ and the stochastic clock $\theta$.
\end{lemma}

%\begin{proof} Because $\tilde{Z}$ is a martingale, it is integrable and as $\tilde{Z}$ is positive, $Z$ is integrable too: $\E[|Z_t|]=\E[|\tilde{Z}_{\theta_t}|]=\E[\E[\tilde{Z}_{\theta_t}|\theta_t]]=\tilde{Z}_0=Z_0<\infty$. Let $\cF_t^{\tilde{Z}}=\sigma(\tilde{Z}_s,0\leq s\leq t)$ and $\cF_t^\theta=\sigma(\theta_s,0\leq s\leq t)$ be the filtrations generated by $\tilde{Z}$ and $\theta$. Denote $\cF^{Z}_{\theta_t}:=\sigma(\tilde{Z}_{\theta_{s}},0\leq s\leq t)$ and $\cF_t:=\cF^{Z}_{\theta_t}\vee \cF_t^\theta$. Then, we have
%
%\beqn
%\E[Z_t|\cF_s]&=&\E[\tilde{Z}_{\theta_t}|\cF_s]\nonumber\\
%&=&\E[\E[\tilde{Z}_{\theta_t}|\cF_s\vee \cF_t^\theta]|\cF_s]\nonumber\\
%&=&\E[\E[\tilde{Z}_{\theta_t}|\cF_{\theta_s}^{\tilde{Z}}\vee \cF_t^\theta]|\cF_s]\nonumber\\
%&=&\E[\tilde{Z}_{\theta_s}|\cF_s]\nonumber\\
%&=&Z_s\nonumber
%\eeqn
%
%where the chain of equalities result from (i) definition of $Z$ (ii) that $\cF_s\subseteq \cF_s\vee \cF_t^\theta$ (iii) $\cF_s\vee \cF_t^\theta = \cF_{\theta_s}^{\tilde{Z}}\vee \cF_t^\theta$ as $\cF_s^\theta \subseteq \cF_t^\theta$ (iv) that $\tilde{Z}$ is a martingale independent of $\theta$ and $\theta_t\geq \theta_s$ and (v) that $\tilde{Z}_{\theta_s}\in\cF_s$ as $\theta_s\in \cF_s^\theta$ and $\tilde{Z}_{\theta_s}\in\sigma(\tilde{Z}_{\theta_u},0\leq u\leq s)$ and using again the definition of $Z$.
%\end{proof}

As suggested in~\cite{Cont04}, one possibility to relax the positivity constraint on $\tilde{Z}$ is to impose an integrability condition on $\tilde{Z}$ only. For instance, uniform integrability of $\tilde{Z}$ is enough in that respect.

\begin{lemma} Let $\tilde{Z}$ be a uniformly integrable martingale relative to its natural filtration. Then $Z_\cdot:=\tilde{Z}_{\theta_\cdot}$ is a martingale in the filtration generated by the time-changed process $\tilde{Z}$ and the stochastic clock $\theta$.
\end{lemma}
\begin{proof} It is enough to discuss the integrability of $Z$ (the conditional expectation discussion is the same as above). The martingale property of $\tilde{Z}$ forces $|\tilde{Z}|$ to be a submartingale: $\E[|\tilde{Z}_t|]\leq\E[|\tilde{Z}_\infty|]$ where the right-hand side is bounded by some constant $M$ from uniform integrability. Hence,
$$\E[|Z_t|]=\E[|\tilde{Z}_{\theta_t}|]\leq \E[|\tilde{Z}_{\infty}|]\leq M\;.$$
\end{proof}

Note that the requirement that $\tilde{Z}$ is integrable on $[0,\infty)$ is needed in the case where $\theta$ is unbounded. One can weaken the condition on $\tilde{Z}$ by moving the integrability requirement on the time-changed process $\theta$ as shown in the below lemma.

\begin{lemma}\label{lemma:TCmartBound} Let $\theta$ be bounded on $[0,T]$ (i.e. there exists an increasing function $k$ such that $\theta_t\leq k(t)$ for all $t$ and thus $\theta_t\leq k(T)$) and $\tilde{Z}$ be a martingale (in its own filtration) on $[0,k(T)]$, independent from $\theta$. Then, $Z$ is a martingale on $[0,\theta_T]$ in the filtration generated by the time-changed process $\tilde{Z}$ and the stochastic clock $\theta$.

\end{lemma}

\begin{proof}
As $\tilde{Z}$ is integrable on $[0,k(T)]$ there exists an increasing function $f$ such that : $\E[ |\tilde{Z}_t| ] \leq f(t) < \infty$ for all $t\in[0,k(T)]$. Hence, for all $t\in[0,T]$:
$$\E[|\tilde{Z}_{\theta_t}|]\leq \E[|\tilde{Z}_{k(T)}|]\leq f(k(T))<\infty$$
\end{proof}

%From a practical point of view, general time-changed processes $\theta$ that are unbounded on $[0,T]$ may cause some problems, especially when $\tilde{Z}$ is governed by a SDE whose analytical solution is unknown or does not exist. %Indeed, to simulate sample paths for $Z$ on $[0,T]$, one needs to simulate sample paths for $\tilde{Z}$ on $[0,\theta_T]$. This is annoying as $\theta_T$ can take arbitrarily large values. 
From a practical point of view, time-changed processes $\theta$ that are unbounded on $[0,T]$ may cause some problems, especially when the transition densities of $\tilde{Z}$ are not explicitly known. In such cases indeed (or when $\tilde{Z}$ needs to be simulated jointly with other processes), sampling paths of $\tilde{Z}$ calls for a discretization scheme, whose error typically increases with the time step. Hence, sampling $Z$ on $[0,T]$ typically requires a fine sampling of $\tilde{Z}$ on $[0,\theta_T]$, leading to prohibitive computational times if $\theta_T$ is allowed to take very large values.% (AAP insiste sur le cote applicatif, donc j'essaie de justifier l'approche sur base du cote pratique, d'ou aussi l'ajout de l'exemple d'application sur les taux de recouvrement). } 
Hence, the class of time-changed processes $\theta$ that are bounded by some function $k$ on $[0,T]$ for any $T<\infty$ whilst preserving analytical tractability proves to be quite interesting. This is of course violated by most of the standard time-change processes (e.g. integrated CIR, Poisson, Gamma, or Compounded Poisson subordinators). A naive alternative consists in capping the later but this would trigger some difficulties. Using $\theta_t=N_t\wedge t$ would mean that $Z=Z_0$ on $[0,1]$ whilst if we choose $\theta_t=J_t\wedge t$ the resulting process may have linear pieces (hence not be piecewise constant). %If for instance $J_t>t$, then $\theta_t$ will increase linearly up to (at least) $t'>t$, $t'=J_t$. 
There exist however simple time-change processes $\theta$ satisfying $\sup_{s\in[0,t]}\theta_s\leq k(t)$ for some functions $k$ bounded on any closed interval and being piecewise constant, having stochastic jumps and having a non-zero possibility to jump in any time set of non-zero measure. Building PWC martingales using such type of processes is the purpose of next section.

% ====================================================================================
\section{Lazy martingales}
% ====================================================================================

We first present a stochastic time-change process that satisfies this condition in the sense that the calendar time is always ahead of the stochastic clock that is, satisfies the boundedness requirement of the above lemma with the linear boundary $k(t)=t$. We then use the later to create PWC martingales. 

% ------------------------------------------------------------------------------------
\subsection{Lazy clocks}
% ------------------------------------------------------------------------------------

We would like to define stochastic clocks that keep time frozen almost everywhere, can jump occasionally, but can't go ahead of the real clock. Those stochastic clocks would then exhibit the piecewise constant path and the last constraint has the nice feature that any stochastic process $Z$ adapated to $\mathbb{F}$, $Z_t\in\cF_t$ is also adapted to $\mathbb{F}$ enlarged with the filtration generated by $\theta$. In particular, we do not need to know the value of $Z$ after the real time $t$. As far as $Z$ is concerned, only the sample paths of $Z$ (in fact $\tilde{Z}$) up to $\theta_t\leq t$ matters. In the sequel, we consider a specific class of such processes, called \textit{lazy clocks} hereafter, that have the specific property that the stochastic clock typically ``sleeps'' (i.e. is ``on hold''), but  gets synchronized to the calendar time at some random times.

\begin{definition} The stochastic process $\theta:\mathbb{R}^+\to\mathbb{R}^+,~t\mapsto \theta_t$ is a $\FF$-\emph{lazy clock} if it satisfies the following properties
\begin{itemize}
\item[$i)$] it is a $\FF$-time change process: in particular, it is grounded ($\theta_0=0$), $\FF$-adapted, c\`adl\`ag and non-decreasing;
\item[$ii)$] it has piecewise constant sample paths : $\theta_t = \sum_{s\leq t} \Delta \theta_s$;
\item[$iii)$] it can jump at any time and, when it does, it synchronizes to the calendar clock.
\end{itemize}
\end{definition}
%\textcolor{red}{ Ici on parle de $\FF$-lazy clock mais $\FF$ n'intervient pas dans la def ! Du coup je ne comprends pas trop la Remark 2. Par exemple ci-dessous si on choisit $\cF_t:=\sigma(A_s,s\leq t)$ alors $\theta_t\in\cF_t$ et $\theta$ est un $\FF$-temps d'arret il me semble.\\
%$\FF$ apparait en fait dans la definition du changement de temps auparavant. Je l'ai remis ici. Si on choisit $\cF_t:=\sigma(A_s,s\leq t)$, alors $\theta_t\in\cF_t$ mais ce n'est a priori pas un temps d'arret. Pour que $\theta_t$ soit un temps d'arret, il faut que pour tout instant $s\leq t$, on soit capable de savoir si $\theta_t$ a eu lieu avant s ou non. Comme on travaille moralement avec des sup, ca n'est pas le cas sauf pour $s=t$, i.e. $\theta_t$ est $\cF_t$-mesurable }

%\textcolor{red}{OK avec tes commentaires et tes changements. Par contre, je me dis qu'on devrait commenter l'echantillonage de $\theta_t$ sur $[0,T]$ car ce n'est pas si simple dans le cas Brownien et Bessel il me semble. Ce que j'ai fait ici c'est tres approximatif: j'ai simplement echantillone un MB wur une grille et identifie les intervalles ou il y a un changement de signe de $W$. Mais ce n'est pas correct (approx discrete + ``rebond'' en 0 negliges). Il faudrait discuter dans la derniere section comment simuler les instants de synchronisation en pratique. La figure de gauche (Poisson) est par contre realisee sur base de la distribution exacte.}
%
\begin{figure}[h!]
\centering
\subfigure[Poisson Lazy clock ($\lambda=3/2$, see Section~\ref{sec:PLC})]{\includegraphics[width=0.48\columnwidth]{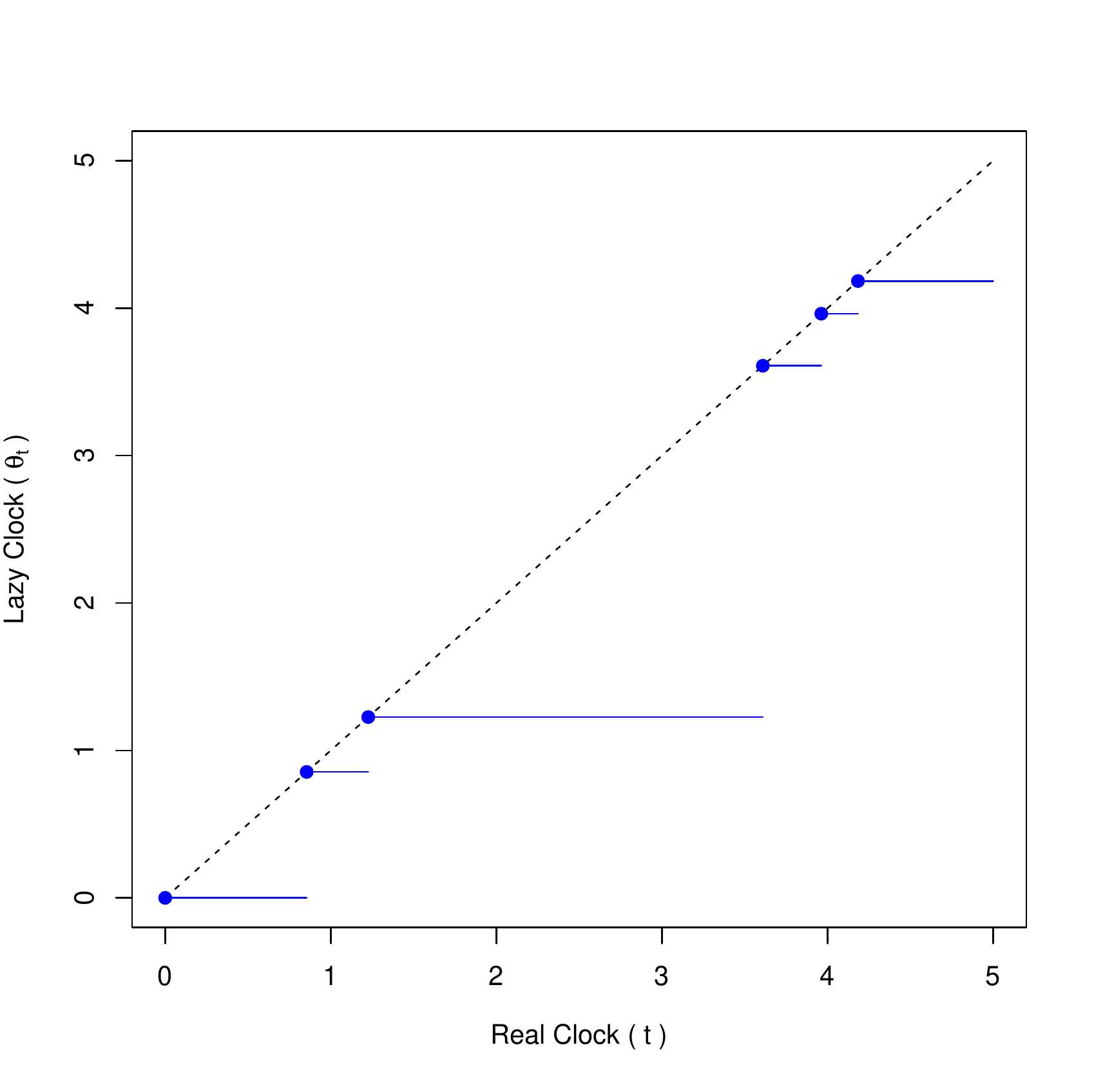}}\hspace{0.2cm}
\subfigure[Brownian Lazy clock (see Section~\ref{sec:BLC})]{\includegraphics[width=0.48\columnwidth]{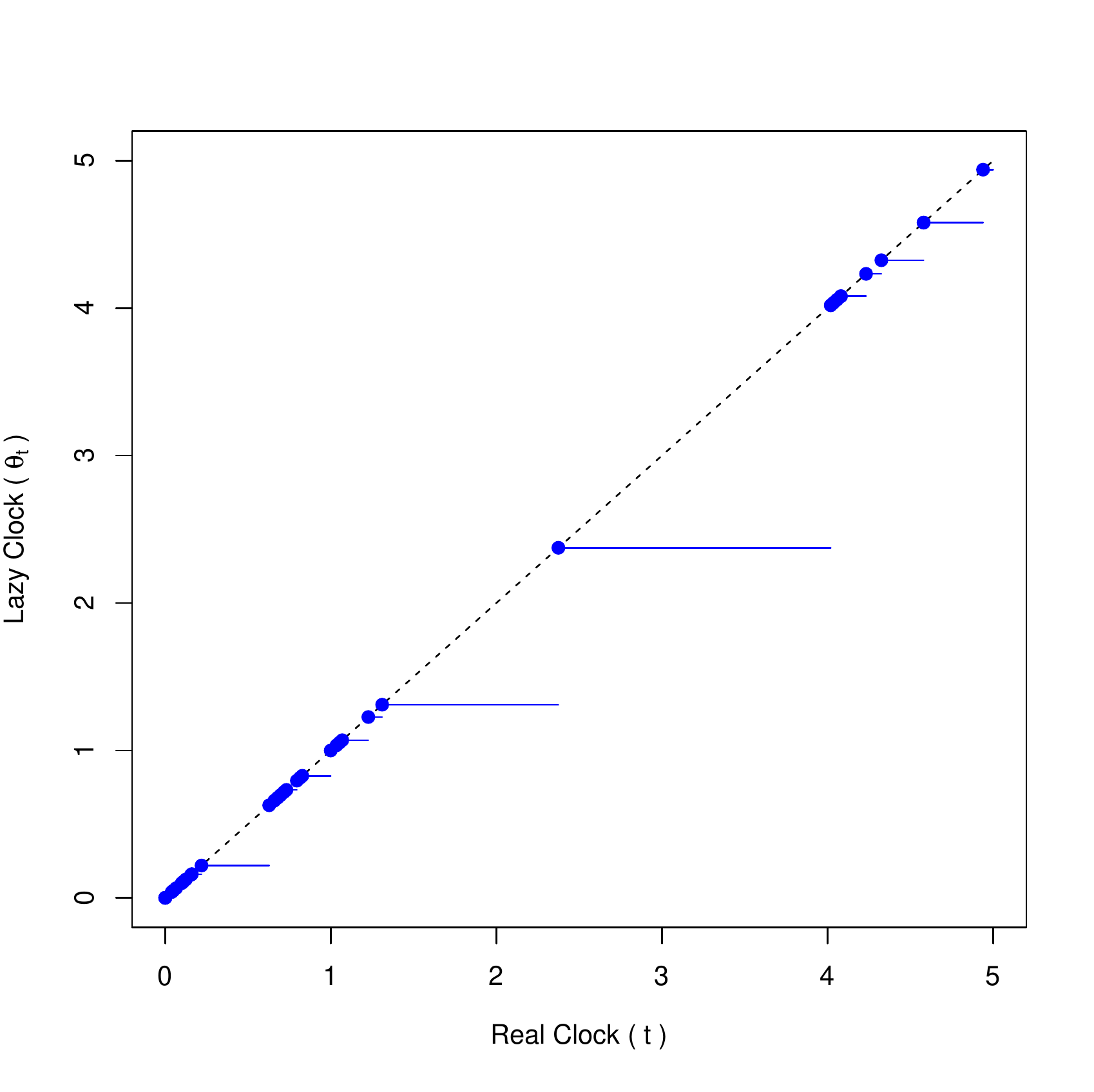}}
\caption{Sample path of Lazy clocks on $[0,5]$.}\label{fig:LC}
\end{figure}

In the sense of this definition, Poisson and Compound Poisson processes are examples of subordinators that keep time frozen almost everywhere but are not lazy clocks however as nothing constraints them to reach $t$ if they jump at $t$. Neither are their capped versions as there are some intervals during which $\theta$ cannot jump or grows linearly. \medskip

\begin{remark}
Note that for each $t>0$, the random variable $\theta_t$ is a priori not a $(\mathcal{F}_s, s\geq0)$-stopping time. In fact, defining 
$$C_t := \inf\{s~;~\theta_s>t\}$$
then $(C_t,\, t\geq0)$ is an increasing family of $\FF$-stopping times. Conversely, for every $t\geq 0$, the lazy clock $\theta$ is a family of $(\mathcal{F}_{C_s},\,s\geq0)$-stopping times, see Revuz-Yor \cite[Chapter V]{RevYor99}.  
\end{remark}

In the following, we shall show that lazy clocks are essentially linked with last passage times, as illustrated in the next proposition.
\medskip

\begin{proposition}
A process $\theta$ is a lazy clock if and only if there
exists a c\`adl\`ag  process $A$ such that the set $\mathcal{Z}:=\{s;\, A_{s^-}=0\text{ or }A_s=0\}$ has a.s. zero Lebesgue measure and $\theta_t=g_t$ with 
$$g_t:=\sup\{s\leq t;\, A_{s^-}=0\text{ or }A_s=0\}\;.$$
\end{proposition}

\begin{proof}
If $\theta$ is a lazy clock, then the result is immediate by taking $A_t=\theta_t-t$ which is c\`adl\`ag, and whose set of zeroes coincides with the jumps of $\theta$, hence is countable. Conversely, fix a path $\omega$. Since $A$ is c\`adl\`ag, the set $\mathcal{Z}(\omega)=\{s; A_{s^-}(\omega)=0\text{ or }A_s(\omega)=0\}$ is closed, hence its complementary may be written as a countable union of disjoint intervals. We claim that 
\begin{equation}\label{eq:Z}
\mathcal{Z}^c(\omega) = \bigcup_{s\geq0} ]g_{s^-}(\omega), g_s(\omega)[\;.
\end{equation}
Indeed, observe first that since $s\longmapsto g_s(\omega)$ is increasing, its has a countable number of discontinuities, hence the union on the right hand side is countable. Furthermore, the intervals which are not empty are such that $A_s(\omega)=0$ or $A_{s^-}(\omega)=0$ and $g_s(\omega)=s$. In particular, if  $s_1<s_2$ are associated with non empty intervals, then 
$g_{s_1}(\omega)=s_1 \leq g_{s_2^-}(\omega)$ which proves that the intervals are disjoint.\\
Now, let $u\in \mathcal{Z}^c(\omega)$. Then $A_u(\omega)\neq 0$. Define $d_u(\omega) = \inf\{s\geq u, \,A_{s^-}(\omega)=0 \text{ or } A_s(\omega)=0\}$. By right-continuity, $d_u(\omega)>u$. We also have $A_{u^-}(\omega)\neq0$ which implies that $g_u(\omega)<u$. Therefore, $u\in ]g_u(\omega), d_u(\omega)[$ which is non empty, and this may also be written $u\in  ]g_{d_u^-(\omega)}(\omega),\, g_{d_u(\omega)}(\omega)[$ which proves the first inclusion. Conversely, it is clear that if $u\in ]g_{s^-}(\omega), g_s(\omega)[$, then $A_u(\omega)\neq0$ and $A_{u^-}(\omega)\neq0$. Otherwise, we would have $u=g_u(\omega)\leq g_{s^-}(\omega)$ which would be a contradiction. Equality (\ref{eq:Z}) is thus proved. Finally, it remains to write :
$$g_t = \int_0^{g_t} 1_{\mathcal{Z}} ds + \int_0^{g_t} 1_{\mathcal{Z}^c} ds = \sum_{s\leq t} \Delta g_s$$
since $\mathcal{Z}$ has zero Lebesgue measure.\\
\end{proof}

We give below examples of lazy clocks admitting simple closed-form distributions.

\subsubsection{Poisson Lazy clock}\label{sec:PLC}

\begin{example}
Let $(X_k, k\geq1)$ be strictly positive random variables and consider the counting process 
 $\left(N_t := \sum_{k=1}^{+\infty} 1_{\{ \sum_{i=1}^k X_i\leq t\}},\, t\geq0\right).$
Then the process $(g_t(N), t\geq0)$ defined as the last jump time of $N$ prior to $t$ or zero if $N$ did not jump by time $t$:
\beq
g_t(N):=\sup\{s\leq t; N_s\neq N_{s^-}\}= \sum_{k=1}^{+\infty} X_k 1_{\{ \sum_{i=1}^k X_i\leq t\}}
\eeq
is a lazy clock.
%For any counting process $N$ we define $\tau_i$ to be the time associated to $i$-th jump of $N$ and set $\tau_0:=0$. We define $g_t(N)$ as the last jump time of $N$ prior to $t$ or zero if $N$ did not jump by time $t$:
%\beq
%g_t(N):=\tau_{N_t}.\label{eq:taus}
%\eeq
\end{example}

%This lazy clock is a particular case of Example \ref{exa:stepclock}. Indeed, setting $X_i=\tau_i - \tau_{i-1}$ for $i\geq 1$, we may observe that 
%$$g_t(N) = \sum_{k=1}^{+\infty} X_k 1_{\{ \sum_{i=1}^k X_i\leq t\}}\qquad \qquad \left(\text{ and }\; N_t = \sum_{k=1}^{+\infty} 1_{\{ \sum_{i=1}^k X_i\leq t\}}\right)$$
 In the case where $N$ is a Poisson process of intensity $\lambda$, i.e. when the r.v.'s  $(X_k,\, k\geq1)$ are i.i.d. with an exponential distribution of parameter $\lambda$, the law of $g_t(N)$ may easily be computed as follows.

\begin{lemma} Assume that $N$ is a Poisson process with parameter $\lambda$. Let $t\geq 0$ and $\delta(x)$ be the Dirac density centered at 0. The distribution of $g_t(N)$ is given by
\beq
f_{g_t(N)}(s)=e^{-\lambda t}(\delta(s)+\lambda e^{\lambda s})~~,~~0\leq s\leq t\label{eq:ThLazyPoisson}
\eeq
and is zero elsewhere. Hence, the cumulative distribution function is
\beq
F_{g_t(N)}(s)=1_{\{0\leq s\leq t\}}e^{-\lambda (t-s)}+1_{\{s>t\}}\label{eq:ThLazyPoissonCDF}
\eeq
and the moments are given by
\beq
\E\left[(g_t(N))^k\right]=\frac{k!}{(-\lambda)^k}(1-\e^{-\lambda t})+\sum_{i=0}^{k-1}(-1)^i\frac{t^{k-i}k!}{\lambda^i (k-i)!}~~,~~k\in \{1,2,\ldots\}
\eeq
%In particular,
%$\E[g_t(N)]=t-(1-e^{-\lambda t})/\lambda$ and $\var[g_t(N)]=(1-\e^{-\lambda t}(2\lambda t+\e^{-\lambda t}))/\lambda^2$.
\end{lemma}
\begin{proof}
%This result was first proven adopting a similar strategy as in Propostion 3 of~\cite{Vrins16b}. 
%Yet, the knowledge of~\eqref{eq:ThLazyPoisson} suggests \textit{a posteriori} the existence of a much easier route, that we now give (the original proof is provided in the appendix). 
This result may be proven adopting a similar strategy as in Propostion 3 of~\cite{Vrins16b}, but we shall take here a shorter route. We merely have to show that (i) $\Pr(g_t(N)=0)=e^{-\lambda t}$, (ii) $f_{g_t(N)}(s)=\lambda e^{-\lambda(t-s)}$ for all $0< s<t$ and (iii) $\Pr(g_t(N)\leq s)=1$ if $s\geq t$. The event $\{g_t(N)=0\}$ is equivalent to $\{N_t=0\}$ whose probability is $e^{-\lambda t}$, proving (i). But $g_t(N)\leq t$ $\Pr$-a.s. justifying (iii). The central point is to notice that the stochastic clock synchronizes to the real clock at each jump. When $t>s$, the event $\{g_t(N)\leq s\}$ is equivalent to say that no synchronization took place after $s$, i.e. $\{N_t=N_s\}$, whose probability is $\Pr(N_{t-s}=0)=e^{-\lambda(t-s)}$. Hence, $g_t(N)$ has a mixed distribution: it is zero for $s<0$ and $s>t$, has a probability mass of $e^{-\lambda t}$ at $s=0$, and a density part of $\lambda e^{-\lambda(t-s)}$ for $s\in (0,t]$; the proof is complete.
\end{proof}

%\textcolor{red}{j'ai enleve le ``drifted'', faire la distinction sur les classes d'horloges (Poisson/Brownien/Bessel) me semble suffisant au niveau des titres}

%\subsubsection{Stable lazy clock}

%Let $L$ be a stable L\'evy process with parameter $\alpha \in (1,2]$. Then, from Bertoin \cite[Chapter VIII, Theorem 12]{Bert96}
%$g_t^{0}(L) = \sup\{s\leq t;\, L_s=0\}$ is a lazy clock whose density is given, for $t>0$, by the generalized Arcsine law :
%$$f_{g_t^0(L)}(s)= \frac{\sin(\pi/\alpha)}{\pi}\, s^{-\frac{1}{\alpha}} \,(t-s)^{\frac{1}{\alpha}-1}~~,~~ 0<s<t.$$
%Note that this Lazy clock is also 1-self-similar.

\subsubsection{Brownian Lazy clock}\label{sec:BLC} %Drifted

%A first possible choice for a lazy clock is to consider the so-called ``last exit time of a Brownian motion $W$ to the boundary $\alpha(t)=t$'', i.e. is the last time $s$ prior to $t$ where a Brownian motion $W$ hits the level given by the calendar time, $g_t^{s}(W)$. 

Another simple example is given by the last passage time of a Brownian motion to zero\footnote{Note that the last passage time of a Brownian motion to the level zero is not clearly identified (unlike the last jump of a Poisson process for instance). As is usual, we define it as the supremum of the passage times of $W$ to that level, i.e. of the supremum of the uncountable set $\{s\leq t, W_s=0\}$, which is  inline with the mathematical definition of $g_t^0(W)$ provided in~\eqref{eq:azema}.}, i.e. $(g_t^0(W), \, t\geq0)$. The initial value of the process is $g_0^0(W)=0$ and the density of $g_t^0(W)$ is given by the \textit{L\'evy's arcsine law} (see e.g.~\cite{Prott05} p.230):
\beq
f_{g_t^0(W)}(s)=\frac{1}{\pi\sqrt{s(t-s)}}~~,~~ 0<s<t
\eeq
and zero otherwise. 
It is also possible to consider several extensions, like the last passage time of $W$ at an affine barrier, $\tilde{g}_t(W):=\sup \{s\leq t~;~ W_s= a+bs\}$. The corresponding density expressed in integral form can be found in ~\cite{Salm88} but can be further simplified with the help of the standard Normal cumulative distribution function $\Phi$ and $\Phi^\prime=\phi$, see \cite{Kaha08} :

\beq
f_{\tilde{g}_t(W)}(s)=\frac{\phi\left(\frac{a+bs}{\sqrt{s}}\right)}{\sqrt{s}}\left(\frac{2}{\sqrt{t-s}}\phi\left(b\sqrt{t-s}\right)+2b\Phi\left(b\sqrt{t-s}\right)-b\right)~~,~~ 0<s<t\;.
\eeq

Observe that $\tilde{g}_t(W)$ is not always well-defined. When $a\neq 0$ indeed, one needs to specify $\tilde{g}_t(W)$ in the cases where $W$ never reaches the barrier before $t$. %\textcolor{red}{When the time-horizon $t$ is fixed (as in~\cite{Kaha08}), a common procedure is to define the first passage time of $W$ before $t$ to be $t$ when $W$ does not reach the barrier on $[0,t]$. As a consequence, the last passage time prior to $t$ would also be $t$ in such cases. This leads to a non-zero probability mass at $t$. However, this setup is not appropriate when the time horizon $t$ is not fixed. Instead, } 
We set, as is usual, $\tilde{g}_0(W):=0$. By doing so, $\tilde{g}$ is adapted to the natural filtration of $W$. In contrasts with $g^0_t(W)$, $\tilde{g}_t(W)$ may have a probability mass at zero, corresponding to the probability of $W$ not to reach the affine barrier prior to $t$. Suppose for instance that $a\geq 0$. Then the event $\{\tilde{g}_t(W)=0\}$ is equivalent to the event $\{W_s< a+bs;\,\forall s\in(0,t]\}$, itself equivalent to $\{\max_{s\in (0,t]}\{W_s-bs\}< a\}$. Hence, the probability mass of $\tilde{g}_t(W)$ at 0 corresponds to the probability for a Brownian motion with drift $-b$ to stay below the threshold $a$, which is known to be (see e.g.~\cite{Shrev04} Corollary 7.2.2)

\beq
\Pr(\tilde{g}_t(W)=0)=\Phi\left(\frac{a+bt}{\sqrt{t}}\right)-\e^{-2ab}\Phi\left(\frac{-a+bt}{\sqrt{t}}\right)\;.\nonumber
\eeq
%Taking $(a,b)=(0,1)$ corresponds to the last passage time of a Brownian motion to the level given by the calendar time. \textcolor{red}{Do we have then $\Pr(g_t^0(W)=0)>0$ ?}
%\textcolor{red}{properties of this $g$ ?? stationary, independent ? Has no probability mass at zero}. See also fir curved boundaries.

Observe that this probability vanishes when $a=b=0$. Hence, one can use $g^0(W)$ or $\tilde{g}(W)$ as a lazy clock, depending on whether we want $\Pr(\theta_t=0)$ to be zero or strictly positive for $t>0$.\\
The moments of $\tilde{g}_t(W)$, $k\in \{1,2,\ldots\}$ read 
\beq
\E\left[\left(\tilde{g}_t(W)\right)^k\right]=\frac{\e^{-ab}}{\pi}\int_0^t \e^{-\frac{b^2}{2}u}\int_0^u \left(\left(k-\frac{1}{2}\right)s^{k-\frac{3}{2}}+\frac{a^2}{2}s^{k-\frac{5}{2}}\right)\e^{-\frac{a^2}{2s}} \frac{ds}{\sqrt{u-s}}du
\eeq
which, in the $a=0$ case, simplifies to
\beq
\E\left[\left(\tilde{g}_t(W)\right)^k\right]=\frac{\Gamma\left(k+\frac{1}{2}\right)}{\sqrt{\pi}(k-1)!}\int_0^t u^{k-1}\e^{-\frac{b^2}{2}u}du\;.
\eeq

\subsubsection{Bessel lazy clock}

\begin{lemma}
Let $R$ denote a Bessel process with index $\nu\in(-1,0)$ started from 0. The probability density of the lazy clock
$g_t^0(R) = \sup\{s\leq t; \, R_s=0\}$ is given, for $0< s<t$, by the generalized Arcsine law :
$$f_{g_t^0(R)}(s) = \frac{1}{\Gamma(|\nu|)\Gamma(1+\nu)} (t-s)^\nu s^{-1-\nu}.$$  
Its moments are given via the representation of Beta functions :
\beq
\E\left[\left(g_t^0(R)\right)^k\right]=\frac{\Gamma(k-\nu)}{\Gamma(k+1)\Gamma(|\nu|)}t^k~~,~~k\in \{1,2,\ldots\}
\eeq
Note that this lazy clock is $1$-self similar.
\end{lemma}

\begin{proof}
We have, using the Markov property and applying Fubini (see \cite{GJY}) :
\begin{align*}
\Pr(g_t^0(R)\leq s) &= \E\left[\Pr_{X_z}(T_0>t-s)\right] \\
&= \int_0^{+\infty}  \frac{2^\nu}{y^{2\nu}\Gamma(|\nu|)} \left( \int_{t-s}^{+\infty} u^{\nu-1}e^{-\frac{y^2}{2u}} du\right) \frac{y^{2\nu+1}}{2^\nu s^{\nu+1}\Gamma(\nu+1)} e^{-\frac{y^2}{2s}} dy \\
&= \frac{1}{\Gamma(|\nu|)\Gamma(\nu+1)}  \int_{t-s}^{+\infty} \frac{u^\nu s^{-\nu}}{u+s} du \\
&=\frac{1}{\Gamma(|\nu|)\Gamma(\nu+1)}\int_{\frac{t}{s}-1}^{+\infty} \frac{r^\nu}{r+1} dr
\end{align*}
after the change of variable $u=rs$. The result then follows  by differentiation. %\textcolor{red}{peut-on développer les moments ? (pour plus decoherence avec les deux autres horloges).}
\end{proof}

\subsection{Time-changed martingales with lazy clocks}
% ------------------------------------------------------------------------------------

In this section we consider a martingale $\tilde{Z}$ whose time is changed with an independent lazy clock to obtain a PWC martingale $Z$. We first show that (in most situations) the lazy clock is adapted to the filtration generated by $Z$. This is done by observing that the knowledge of $\theta$ amounts to the knowledge of its jump times, since the size of the jumps are always obtained as a difference with the calendar time. In particular, the properties of the lazy clock allow one to reconstruct  the trajectories of $Z$ on $[0,t]$ only from past values of $\tilde{Z}$ and $\theta$; no information about the future (measured according to the real clock) is required. We then provide the resulting distribution when the clock $g(N)$ is governed by Poisson, inhomogeneous Poisson or Cox processes.

\begin{lemma} Let $\tilde{Z}$ be a stochastic process independent from the lazy clock $\theta$ and assume that $\forall u\neq v,\; \Pr(\tilde{Z}_u =\tilde{Z}_v )=0$. Then, $\theta$ is adapted to the filtration $(\cF^Z_t, \, t\geq0)$. 
\end{lemma}
\begin{proof} Observe first that the countable union
$$\mathcal{N} = \bigcup_{s\leq t, \theta_s=s} \{ Z_{s}= Z_{s^-}  \} = \bigcup_{s\leq t, \theta_s=s} \{ \tilde{Z}_{\theta_s}= \tilde{Z}_{\theta_{s^-}}  \} $$
is of measure zero since $\tilde{Z}$ and $\theta$ are independent. This implies that a.s., the sample paths of $\theta$ (both the jump times and the jump sizes) can be recovered from the sample paths of $Z$ up to $\theta_t$, hence up to $t$. Indeed, the set of the jump times of $\theta$ on $[0,t]$ is given by $\{s\in [0,t]:\theta_{s}=s\}$. Moreover, the ``synchronization events'' $\{\theta_s=s\}$ coincide with the ``jump events'' $\{Z_{s}-Z_{{s^-}}>0\}$ so that all jump times of $\theta$ are identified by the jumps of $Z$. 
But $\theta$ is constant between two jumps and jumps to a known value (the calendar time) each time $Z$ jumps, so we have the a.s. representation $\theta_t = \sup\{s\leq t;\;Z_s\neq Z_{s^-} \}$. This means that both $\theta_t$ and $\tilde{Z}_{\theta_t}$ are revealed in $\cF^Z_{\theta_{t}}$ and, in particular, $\cF^\theta_t\subseteq\cF^Z_{\theta_{t}}$. The proof is concluded by noting that $\theta_t\leq t$, leading to $\cF^Z_{\theta_{t}}\subseteq\cF^Z_{t}$. 
%Events that happen with null probability (e.g. $Z$-jumps of size 0) are dealt with via the ``usual conditions'', i.e. assuming that the filtration is completed with all $\Pr$-null sets.\\
\end{proof}

\begin{lemma}
Let $\tilde{Z}$ be a martingale and $N$ an independent Poisson process with intensity $\lambda$. Then $Z_t:=\tilde{Z}_{g_t(N)}$ is a PWC martingale with same range as $\tilde{Z}$. Its cumulative distribution function is given by :
%Then, $Z$ defined as $Z_t:=\tilde{Z}_{g_t(N)}$ is a right-continuous pure jump process with same range as $Z$ and piecewise constant sample paths. Moreover, it is a pure jump martingale with respect to its own filtration and its cumulative distribution function  is given by
\beq
F_{Z_t}(z)=\Pr(Z_t\leq z)=e^{-\lambda t}\left(1_{\{Z_0\leq z\}}+ \lambda \int_{0}^t F_{\tilde{Z}_u}( z) e^{\lambda u}du\right)
\eeq
\end{lemma}
\begin{proof} This result is obvious from the independence assumption between $\tilde{Z}$ and $N$ (i.e. $\theta=g(N)$),
\beq
F_{Z_t}(z)=\int_0^\infty F_{\tilde{Z}_u}(z) \Pr(\theta_t\in du)\;.\label{eq:FZGivenTheta}
\eeq
\end{proof}

A similar result applies to the inhomogeneous Poisson and Cox cases. The proofs are very similar. 

\begin{corollary} Let $N$ be an inhomogeneous Poisson processes, with (deterministic) intensity $(\lambda(u), u\in[0,T])$ and $\Lambda(t) =  \int_0^t \lambda(u)du$. Then we have : 
\beq
F_{Z_t}(z)=e^{-\Lambda(t)}\left(1_{\{Z_0\leq z\}}+ \int_{0}^t \lambda(u) F_{\tilde{Z}_u}(z) e^{\Lambda(u)}du\right)
\eeq
%\end{corollary}
%\begin{corollary}
In the case where $N$ is an inhomogeneous Poisson process with stochastic intensity (i.e. Cox process) independent from $\tilde{Z}$,
\beq
F_{Z_t}(z)=\left(1_{\{Z_0\leq z\}}P(0,t)+ \int_{0}^t F_{\tilde{Z}_s}( z) d_sP(s,t)\right)\;,
\eeq
 where we have set $P(s,t):=\E[e^{-(\Lambda_t-\Lambda_s)}]$ with $\Lambda_t:=\int_0^t\lambda_u du$ the integrated intensity process.
\end{corollary}
\begin{proof} We start from the inhomogeneous Poisson case, set as hazard rate function $\lambda(u)$ for all $u\in[0,T]$ a sample path $\lambda_u(\omega)$ of the stochastic intensity and take the expectation, which amounts to replace $\lambda(u)$ by $\lambda_u$ (hence $\Lambda(u)$ by $\Lambda_u$) and take the expected value of the resulting cumulative distribution function derived above with respect to the intensity paths:
\beqn
F_{Z_t}(z)&=&\E\left[\E\left[\Pr(Z_t\leq z)|\lambda(u)=\lambda_u,~0\leq u\leq t\right]\right]\\
&=&1_{\{Z_0\leq z\}}\E\left[e^{-\Lambda_t}\right]+ \E\left[\int_{0}^t \lambda_s F_{\tilde{Z}_s}( z)e^{-(\Lambda_t-\Lambda_s)}ds\right]\\
&=&1_{\{Z_0\leq z\}}P(0,t)+ \int_{0}^t F_{\tilde{Z}_s}(z) \E\left[\lambda_se^{-(\Lambda_t-\Lambda_s)}\right]ds
\eeqn
where in the last equality we have used Tonelli's theorem to exchange the integral and expectation operators when applied to non-negative functions as well as independence between $\lambda$ and $\tilde{Z}$.\\
From Leibniz rule, $\lambda_s e^{-(\Lambda_t-\Lambda_s)}=\frac{d}{ds} e^{-(\Lambda_t-\Lambda_s)}$ so
\beq
\E\left[\lambda_s e^{-(\Lambda_t-\Lambda_s)}\right]=\frac{d}{ds}\E\left[ e^{-(\Lambda_t-\Lambda_s)}\right]=\frac{d}{ds}P(s,t)\;.
\eeq
\end{proof}

\begin{remark} Notice that $P(s,t)$ does not correspond to the expectation of $e^{-\int_s^t \lambda_u du}$ conditional upon $\cF_s$, the filtration generated by $\lambda$ up to $s$ as often the case e.g. in mathematical finance. It is an unconditional expectation that can be evaluated with the help of the tower law. In the specific case where $\lambda$ is an affine process for example, $\E\left[e^{-\int_s^t \lambda_u du}|\lambda_s=x\right]$ takes the form $A(s,t)e^{-B(s,t)x}$ for some deterministic functions $A,~B$ so that

$$P(s,t)=\E\left[e^{-\int_s^t \lambda_u du}\right]=\E\left[\E\left[A(s,t)e^{-B(s,t)\lambda_s}\right]\right]=A(s,t)\varphi_{\lambda_s}(iB(s,t))\;.$$
%where $\varphi_X(u):=\E[e^{iuX}]$ is the characteristic function of the random variable $X$. 
\end{remark}

\begin{example}
In the case $\lambda$ follows a CIR process, $d\lambda_t=k(\theta-\lambda_t)dt+\sigma\sqrt{\lambda_t}dW_t$ with $\lambda_0>0$ then $\lambda_s\sim r_s/c_s$ with $c_s=\nu/(\theta(1-e^{-ks}))$ and $r_s$ is a non-central chi-squared random variable with non-centrality parameter $\nu=4k\theta/\sigma^2$ and $\kappa=c_s\lambda_0e^{-ks}$ the degrees of freedom. So, $\varphi_{\lambda_s}(u)=\E[\e^{i(u/c_s) r_s}]=\varphi_{r_s}(u/c_s)$ where $\varphi_{r_s}(v)=\frac{1}{(1-2iv)^{\kappa/2}}\exp\left(\frac{\nu iv}{1-2iv}\right)$. 
\end{example}

% ------------------------------------------------------------------------------------
\subsection{Some Lazy martingales without independence assumption}
% ------------------------------------------------------------------------------------

We have seen that when $\tilde{Z}$ is a martingale and $\theta$ an independent lazy clock, then $(Z_t=\tilde{Z}_{\theta_t}, \, t\geq0)$ is a PWC martingale. We now give an example where the lazy time-change $\theta$ is not independent from the latent process $\tilde{Z}$.
\begin{proposition}
Let $B$ and $W$ be two Brownian motions with correlated coefficient $\rho$ and $f$ a continuous function. Define the lazy clock : %\textcolor{red}{I'd rather use $g_t^f(W)$}
$$g_t^f(W) := \sup\{s\leq t,\; W_s=f(s)\}\;.$$
Let $h(W)$ be a progressively measurable process with respect to $W$ and assume that there exists a deterministic function $\psi$ such that :
$$\int_0^{g_t^f(W)} h_u(W) dW_u = \psi(g_t^f(W))\;.$$
Then, the process $Z_t:=\tilde{Z}_{g_t^f(W)}$ where $\tilde{Z}_t:=\int_0^{t} h_u(W) dB_u - \rho\psi(t)$is a PWC martingale.%\left(\int_0^{g_t^f(W)} h_u(W) dB_u- \rho\psi(g_t^f(W)),\; t\geq0\right)$ is a Lazy martingale. 
%Let $B$ and $W$ be two correlated Brownian motions with coefficient $\rho$ and define the lazy clock :
%$$g_0^{(t)} = \sup\{s\leq t,\; W_s=0\}.$$
%Then, $(B_{g_0^{(t)}}, \,t\geq0)$ is a  $\frac{1}{2}$-self-similar PWC martingale whose probability density is given by :
%$$\Pr(B_{g_0^{(t)}}\in dz) = \frac{1}{\pi \sqrt{2\pi (1-\rho^2) t} } e^{- \frac{z^2}{4t(1-\rho^2)}} K_0\left(\frac{z^2}{4t(1-\rho^2)}  \right)  \, dz$$
\end{proposition}
\begin{proof}
Let $W^\perp$ be a Brownian motion independent from $W$ such that $B = \rho W+\sqrt{1-\rho^2}\,W^\perp $. 
The time-change yields :
\begin{align*}
\int_0^{g_t^f(W)} h_u(W) dB_u - \rho\psi(g_t^f(W)) &=  \int_0^{g_t^f(W)} h_u(W) dB_u -
\rho \int_0^{g_t^f(W)} h_u(W ) dW_u \\
&= \sqrt{1-\rho^2}\int_0^{g_t^f(W)} h_u(B) dW^\perp_u \\
&=   \sqrt{1-\rho^2} W^\perp_{\int_0^{g_t^f(W)} h^2_u(B)du}
\end{align*}
%$$ B_{g_0^{(t)}} -\rho f(g_0^{(t)}) =B_{g_0^{(t)}}- \rho W_{g_0^{(t)}} = \sqrt{1-\rho^2}\,Z_{g_0^{(t)}} $$ 
which is a PWC martingale since $g^f(W)$ and $h(B)$ are  independent from $W^\perp$.
\end{proof}
It is interesting to point out here that the latent process  $\tilde{Z}$ is, in general, \emph{not} a martingale (not even a local martingale). It becomes a martingale thanks to the lazy time-change. % = \int_0^{t} h_u(W) dB_u - \rho\psi(t)

\begin{example}
We give below several examples of application of this  proposition.
\begin{enumerate}
\item Take $h_u=1$. Then, $\psi=f$ and $\left(B_{g_t^f(W)} -\rho f(g_t^f(W)), \,t\geq0\right)$ is a PWC martingale.\\
More generally, we may observe from the proof above that if $H$ is a space-time harmonic function (i.e. $(t,z)\rightarrow H(t,z)$ is $\mathcal{C}^{1,2}$ and such that $\frac{\partial H}{\partial t} + \frac{1}{2} \frac{\partial^2 H}{\partial z^2} =0$), then the process $$\left(H\left(B_{g_t^f(W)}-\rho f(g_t^f(W)),\, (1-\rho^2) g_t^f(W)\right),\; t\geq0\right)$$ is a PWC martingale. %Observe in particular that the latent process here is not, in itself, a martingale.
\item Following the same idea, take $h_u(W) = \frac{\partial H}{\partial z}(W_u,u)$ for some harmonic function $H$. Then 
\begin{align*}\int_0^{g_t^f(W)} \frac{\partial H}{\partial z}(W_u,u) dW_u &= H\left(W_{g_t^f(W)}, g_t^f(W)\right)-H(0,0) \\&=  H\left(f(g_t^f(W)), g_t^f(W)\right)-H(0,0)
\end{align*}
and the process $\left(\int_0^{g_t^f(W)} \frac{\partial H}{\partial z}(W_u,u) dB_u -\rho H\left(f(g_t^f(W)),g_t^f(W)\right) ,\,t\geq0\right)$ is a PWC martingale.  
\item As a last example, take $f=0$ and $h_u=r(L^0_u)$ where $r$ is a $\mathcal{C}^1$ function and $L^0$ denotes the local time of $W$ at 0. Then, integrating by parts :
$$\int_0^{g_t^f(W)} r(L^0_u) dW_u = r(L_{g_t^f(W)}^0) W_{g_t^f(W)} - \int_0^{g_t^f(W)} W_u r^\prime(L^0_u) dL^0_u = 0$$
since the support of $dL^0$ is included in $\{u, W_u=0\}$. Therefore, the process $\left(\int_0^{g_t^f(W)} r(L^0_u)dB_u, \, t\geq0\right)$ is a PWC martingale.
\end{enumerate}
\end{example}

\noindent
\section{Numerical simulations}
%The next examples are direct consequences of the above theorem. They provide pure jump martingales with piecewise constant sample paths taking values on $\mathbb{R}$, $\mathbb{R}^+$ or $(0,1)$ by sampling transforms of a Brownian motion with an independent Poisson process. As the CDF of those transforms admit a closed form expression, the resulting pure jump martingales are analytical up to a time-integral.

In this section, we briefly sketch the construction schemes to sample paths of the lazy clocks discussed above. These procedures have been used to generate Fig.~\ref{fig:LC}. Finally, we illustrate sample paths and distributions of a specific martingale in $[0,1]$ time-changed with a Poisson lazy clock.

\subsection{Sampling of lazy clock and lazy martingales}

By definition, the number of jumps of a lazy clock $\theta$ on $[0,T]$ is countable, but may be infinite. Therefore, except in some specific cases (such as the Poisson lazy clock), an exact simulation is impossible. Using a discretization grid, the simulated trajectories of a lazy clock $\theta$ on $[0,T]$ will take the form
$$\theta_{t}:=\sup\{\tau_i,\tau_i\leq t\}$$
where $\tau_0:=0$ and $\tau_1,\tau_2,\ldots$ are (some of) the synchronization times of the lazy clock up to time $T$. We can thus focus on the sampling times $\tau_1,\tau_2\ldots$ whose values are no greater than $T$. 

\subsubsection*{Poisson lazy clock}
Trajectories of a Poisson lazy clock $\theta_t(\omega)=g_t(N(\omega))$ on a fixed interval $[0,T]$ are very easy to obtain thanks to the properties of Poisson jump times.\medskip

\fbox{%
    \parbox{0.9\textwidth}{		
		\begin{minipage}[c]{0.9\columnwidth}
\vspace{0.2cm}
%\noindent
\begin{algorithm}[Sampling of a Poisson lazy clock]~~~~~~~~~~~~~~~~~~~~~~
\begin{enumerate}
\item Draw a sample $n=N_T(\omega)$ for the number of jump times of $N$ up to $T$: $N_T\sim Poi(\lambda T)$.
\item Draw $n$ i.i.d. samples from a standard uniform $(0,1)$ random variable $u_i=U_i(\omega)$, $i\in \{1,2,\ldots,n\}$  sorted in increasing order $u_{(1)}\leq u_{(2)}\leq \ldots \leq u_{(n)} $.
\item Set $\tau_i:=Tu_{(i)}$ for $i\in \{1,2,\ldots,n\}$.
\end{enumerate}
\end{algorithm}		   
\end{minipage}\vspace{0.2cm}
}
}

\subsubsection*{Brownian lazy clock}

Sampling a trajectory for a Brownian lazy clock requires the last zero of a Brownian bridge. This is the purpose of the following lemma.

\begin{lemma}
Let $W^{x,y,t}$ be a Brownian bridge on $[0,t]\;, t\leq T$, starting at  $W_0^{x,y,t}=x$ and ending $W_t^{x,y,t}=y$, and define its last passage time at 0 :
$$g_t(W^{x,y,t}):= \sup\{s\leq t,\; W^{x,y,t}_s=0\}.$$
Then, the cumulative distribution function $F(x,y,t; s)$ of $g_t(W^{x,y,t})$ is given, for $s\in[0,t]$ by :
\beqn
\Pr(g_t(W^{x,y,t})\leq s)=F(x,y,t;s)&:=&1-\e^{-\frac{xy}{t}}\left(d_{+}(x,y,t;s)+d_{-}(x,y,t;s)\right)\;,\nonumber\\
\text{where }\qquad d_{\pm}(x,y,t; s)&:=&\e^{\frac{\pm|xy|}{t}}\Phi\left(\mp|x|\sqrt{\frac{t-s}{st}}-|y|\sqrt{\frac{s}{t(t-s)}}\right)\;.\nonumber
\eeqn
In particular, the probability that $W^{x,y,t}$ does not hit 0  during $[0,t]$ equals:
$$\Pr(g_t(W^{x,y,t})=0)=F(x,y,t;0) = 1-e^{-\frac{xy + |xy|}{t}}.$$
Note also the special case when $y=0$ :
$$\Pr(g_t(W^{x,0,t})=t)=1.$$
\end{lemma}
%\textcolor{red}{preuve ? Je propose qu'on donne la version ``integrale'', plus courte a montrer, et de se contenter de donner les etapes pour aboutir a la forme close.}
\begin{proof}
Using time reversion and the absolute continuity formula of the Brownian bridge with respect to the free Brownian motion (see Salminen \cite{Sal97}), the density of $g_t(W^{x,y,t})$ is given, for $y\neq0$, by :
$$
\Pr(g_t(W^{x,y,t})\in ds) = \frac{|y|\sqrt{t}}{\sqrt{2\pi}} e^{\frac{(y-x)^2}{2t}} \frac{1}{\sqrt{s}(t-s)^{3/2}}e^{-\frac{x^2}{2s}}e^{-\frac{y^2}{2(t-s)}}\, ds.
$$
Integrating over $[0,t]$, we first deduce that 
\begin{equation}\label{eq:Ftt}
\frac{|y|\sqrt{t}}{\sqrt{2\pi}} \int_0^t  \frac{e^{-\frac{x^2}{2s}}}{\sqrt{s}} \frac{e^{-\frac{y^2}{2(t-s)}}}{(t-s)^{3/2}}\, ds = \exp\left(\frac{(|y|+|x|)^2}{2t}\right).
\end{equation}
We shall now compute a modified Laplace transform of $F$, and then invert it. Integrating by parts and using (\ref{eq:Ftt}), we deduce that :
$$
\lambda \int_0^t \frac{e^{-\frac{\lambda}{2s}}}{2s^2} F(x,y,t; s) ds = e^{-\frac{\lambda}{2t}}  - e^{-\frac{\lambda}{2t}} \exp\left(-\frac{xy}{t} -\frac{|y| \sqrt{\lambda+x^2}}{t}\right).
$$
Observe next that by a change of variable :
$$\lambda \int_0^t \frac{e^{-\frac{\lambda}{2s}}}{2s^2} F(x,y,t; s) ds  = \lambda e^{-\frac{\lambda}{2t}}\int_0^{+\infty} e^{-\lambda v} F\left(x,y,t; \frac{1}{2v+1/t}\right) dv $$
hence
$$\int_0^{+\infty} e^{-\lambda v} F\left(x,y,t; \frac{1}{2v+1/t}\right) dv = \frac{1}{\lambda}  - \frac{1}{\lambda} \exp\left(-\frac{xy}{t} -\frac{|y| \sqrt{\lambda+x^2}}{t}\right)$$
and the result follows by inverting this Laplace transform thanks to the formulae, for $a>0$ and $b>0$ :
$$\frac{1}{\lambda}\exp\left(-a\sqrt{\lambda+x^2}\right) = \frac{a}{2\sqrt{\pi}} \int_0^{+\infty} e^{-\lambda v} \int_0^v e^{-ux^2} \frac{1}{u^{3/2}} e^{-\frac{a^2}{4u}}du\, dv$$
and
$$\int_0^z e^{-au - b/u} \frac{du}{u^{3/2}} = \frac{\sqrt{\pi}}{2\sqrt{b}}\left(e^{2\sqrt{ab}}\text{Erfc}\left(\sqrt{\frac{b}{z}} - \sqrt{az}\right)+e^{-2\sqrt{ab}}\text{Erfc}\left(\sqrt{\frac{b}{z}} + \sqrt{az}\right) \right).$$
\end{proof}

%\begin{lemma} Let $W^{x,y,t}$ be a Brownian motion on $[0,t]\;, t\leq T$ satisfying $W_0=x$ and $W_t=y$. Then, the probability that $W^{x,y,t}$ has no zero prior to $t$ is $F(0;x,y,t)$ where, for $0\leq s\leq t$
%\beqn
%F(s;x,y,t):=1-\int_0^{t-s} \sqrt{\frac{t}{t-u}}\e^{-\frac{x^2}{2(t-u)}}\e^{\frac{(y-x)^2}{2t}}\frac{|y|}{\sqrt{2\pi u^3}}\e^{\frac{-y^2}{2u}}du\;.
%F(s;x,y,t)&:=&1-\e^{-\frac{xy}{t}}\left(d_{+}(x,y,s,t)+d_{-}(x,y,s,t)\right)\;,\\
%d_{\pm}(x,y,s,t)&:=&\e^{\frac{\pm|xy|}{t}}\Phi\left(\mp|x|\sqrt{\frac{t-s}{st}}-|y|\sqrt{\frac{s}{t(t-s)}}\right)\;.
%\eeqn
%Moreover, the distribution function of $g_t(W^{x,y,t})$ conditional on the fact that $W$ reaches 0 prior to $t$ reads
%\[
%F_{g_t(W;x,y)}(s):=\frac{F(s;x,y,t)-F(0;x,y,t)}{1-F(0;x,y,t)}\;.
%\]
%\end{lemma}
Simulating a continuous trajectory of a Brownian lazy clock $\theta$ in a perfect way is an impossible task. The reason is that when a Brownian motion reaches zero at a specific time say $s$, it does so infinitely many times on $(s,s+\varepsilon]$ for all $\varepsilon>0$. Consequently, it is impossible to depict such trajectories in a perfect way. Just like for the Brownian motion, one could only hope to sample trajectories on a discrete time grid, where the maximum stepsize provides some control about the approximation, and corresponds to a basic unit of time. By doing so, we disregard the specific jump times of $\theta$, but focus on the supremum of the zeroes of a Brownian motion in these intervals. To do this, we proceed as follows.\medskip

\fbox{%
    \parbox{0.9\textwidth}{		
		\begin{minipage}[c]{0.9\columnwidth}
\vspace{0.2cm}
%\noindent
\begin{algorithm}[Sampling of a Brownian lazy clock]~~~~~~~~~~~~~~~~~~~~~~
\begin{enumerate}
\item Fix a number of steps $n$ such that time step $\delta=T/n$ corresponds to the desired time unit.
\item Sample a Brownian motion $w=W(\omega)$ on the discrete grid $[0,\delta,2\delta,\ldots,n\delta]$.
\item In each interval $((i-1)\delta,i\delta]$, $i\in\{ 1,2,\ldots,n\}$, draw a uniform $(0,1)$ random variable $u_i=U_i(\omega)$
\begin{itemize}
\item If $u_i< F(w_{(i-1)\delta},w_{i\delta},\delta;0)$ then $w$ does not reach 0 on the interval 
\item Otherwise, set the supremum $g_i$ of the last zero of $w$ as the $s$-root of $F(w_{(i-1)\delta},w_{i\delta},\delta;s)-u_i$
\end{itemize} 
\item Identify the $k$ intervals ($1\leq k\leq n$) in which $w$ has a zero, and set  $\tau_j:=i_j\delta+g_{i_j}$, $j\in\{1,\ldots,k\}$ where $i_j\delta$ is the left bound of the interval.
\end{enumerate}

\end{algorithm}		   
\end{minipage}\vspace{0.2cm}
}
}

\subsection{Example: $\Phi$-martingale sampled with a Poisson lazy clock}
We conclude this note by providing simulations of a PWC martingale in $[0,1]$ (as well as its probability distribution) obtained by sampling the ``so-called'' $\Phi$-martingale with a Poisson lazy clock.  Lazy martingales evolving in $\mathbb{R}$ (resp. $\mathbb{R}^+$) can be found in a similar way by considering a Brownian motion $ W$ (resp. Dol\'eans-Dade exponential $\cE(W)$) as latent process $\tilde{Z}$. The resulting expressions are equally tractable.

%\begin{example}[PWC martingale on $\mathbb{R}$]
%Let $N$ be a Poisson process with intensity $\lambda$, $W$ be a Brownian motion and
%\beq
%\tilde{Z_t}:=\sigma W_t\;.
%\eeq
%Then, the stochastic process $Z$ defined as $Z_t:=\tilde{Z}_{g_t(N)}$, $t\geq 0$, is PWC martingale on $\mathbb{R}$ with CDF
%\beq
%F_{Z_t}(z)=\Pr(Z_t\leq z)=e^{-\lambda t}\left(\ind_{\{Z_0\leq z\}}+ \lambda \int_{0}^t \Phi\left(\frac{z}{\sigma \sqrt{s}}\right) e^{\lambda s}ds\right)\;.
%\eeq
%\end{example}

%\begin{example}[PWC martingale on $\mathbb{R}^+$]
%Let $N$ be a Poisson process with intensity $\lambda$ and $\tilde{Z}=\cE(\sigma W)$ be the Dol\'eans-Dade exponential of $\sigma W$, 
%
%\beq
%\tilde{Z}_{t}:=Z_0\exp\left\{\sigma W_t-\frac{\sigma^2}{2} t\right\}\;.
%\eeq
%Then, the stochastic process $Z$ defined as $Z_t:=\tilde{Z}_{g_t(N)}$, $t\geq 0$, is a pure jump martingale on $\mathbb{R}^+$ with CDF
%
%\beq
%F_{Z_t}(z)=e^{-\lambda t}\left(1_{\{Z_0\leq z\}}+ \lambda \int_{0}^t \Phi\left(\frac{\log(z)+\frac{\sigma^2}{2} s}{\sigma \sqrt{s}}\right) e^{\lambda s}ds\right)\;.
%\eeq
%\end{example}

\begin{example}[PWC martingale on $(0,1)$]
Let $N$ be a Poisson process with intensity $\lambda$ and $\tilde{Z}$ be the $\Phi$-martingale \cite{Vrins16} with constant diffusion coefficient $\eta$,

\beq
\tilde{Z}_{t}:=\Phi\left(\Phi^{-1}(Z_0)e^{\eta^2/2 t}+\eta \int_0^{t}e^{\frac{\eta^2}{2}(t-s)}dW_s\right)\;.
\eeq
where $\Phi$ denotes as before the standard Normal CDF. Then, the stochastic process $Z$ defined as $Z_t:=\tilde{Z}_{g_t(N)}$, $t\geq 0$, is a PWC martingale on $(0,1)$ with CDF
\beq
F_{Z_t}(z)=e^{-\lambda t}\left(1_{\{Z_0\leq z\}}+ \lambda \int_{0}^t \Phi\left(\frac{\Phi^{-1}(z)-\Phi^{-1}(Z_0)e^{\eta^2/2 u}}{\sqrt{e^{\eta^2u}-1}}\right) e^{\lambda u}du\right)\;.
\eeq
\end{example}

Some sample paths for $\tilde{Z}$ and $Z$ are drawn on Fig.~\ref{fig:SP}. Notice that this martingale $\tilde{Z}$ can be simulated without error using the exact solution.
%Notice that all the martingales $\tilde{Z}$ given above can be simulated without error using the exact solution.

%
\begin{figure}[h!]
\centering
\subfigure[$\eta=25\%$, $\lambda=20\%$]{\includegraphics[width=0.48\columnwidth]{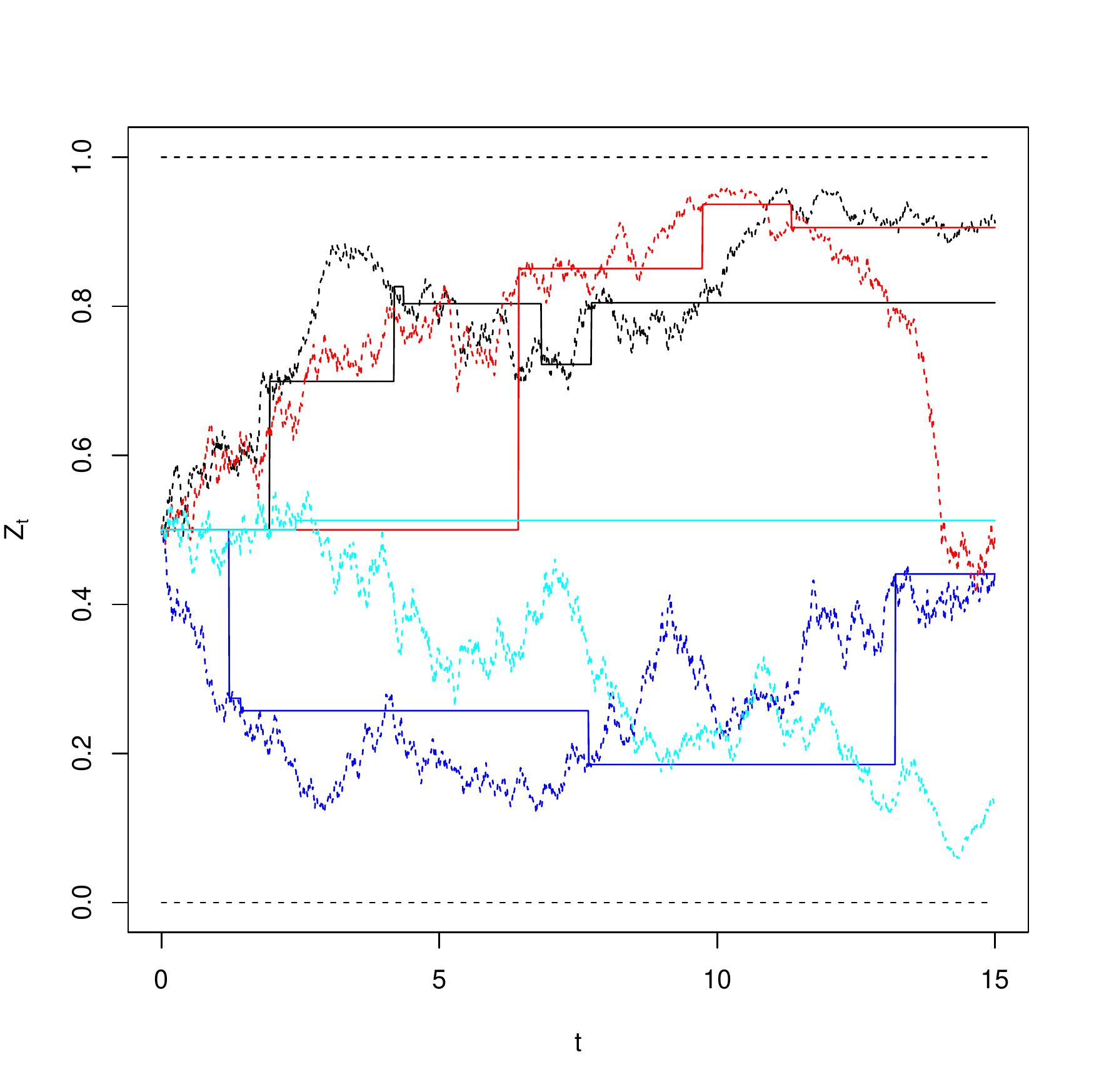}}\hspace{0.2cm}
\subfigure[$\eta=15\%$, $\lambda=50\%$]{\includegraphics[width=0.48\columnwidth]{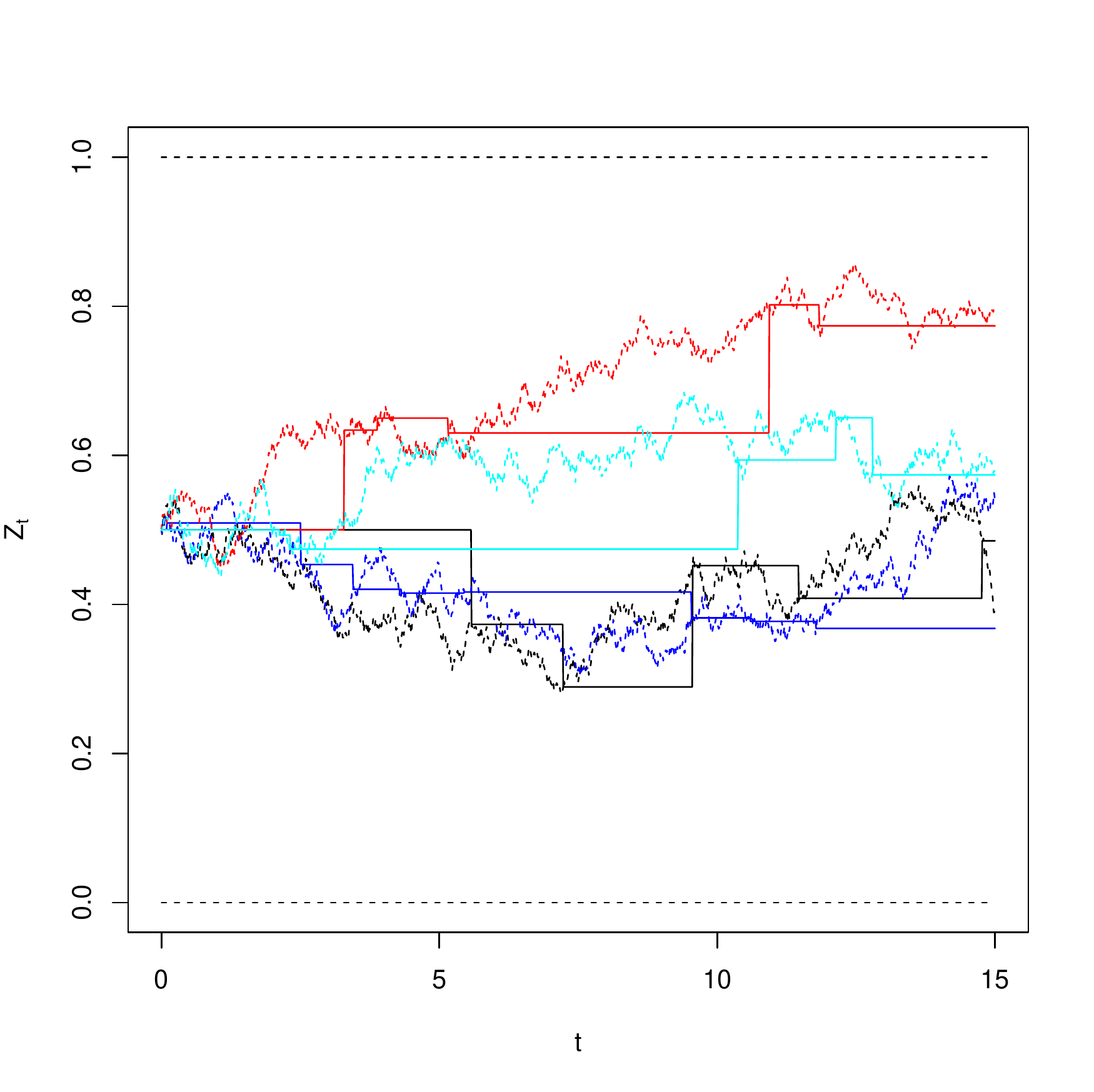}}
\caption{Four sample paths of $\tilde{Z}$ (circles) and $Z$ (no marker) up to $T=15$ years where $\tilde{Z}$ is the $\Phi$-martingale with $Z_0=0.5$.}\label{fig:SP}
\end{figure}

Figure~\ref{fig:CDF} gives the cumulative distribution function of $Z$ and $\tilde{Z}$ where the later is a $\Phi$-martingale. The main differences between these two sets of curves result from the fact that $\Pr(\tilde{Z}_t=Z_0)=0$ for all $t>0$ while $\Pr(Z_t=Z_0)=\Pr(\tilde{Z}_{g_t(N)}=Z_0)=\Pr(N_t=0)>0$ and that there is a delay resulting from the fact that $Z_t$ correspond to some past value of $\tilde{Z}$.
\begin{figure}[h!]
\centering
\subfigure[$Z_0=50\%$, $\eta=25\%$, $\lambda=20\%$]{\includegraphics[width=0.48\columnwidth]{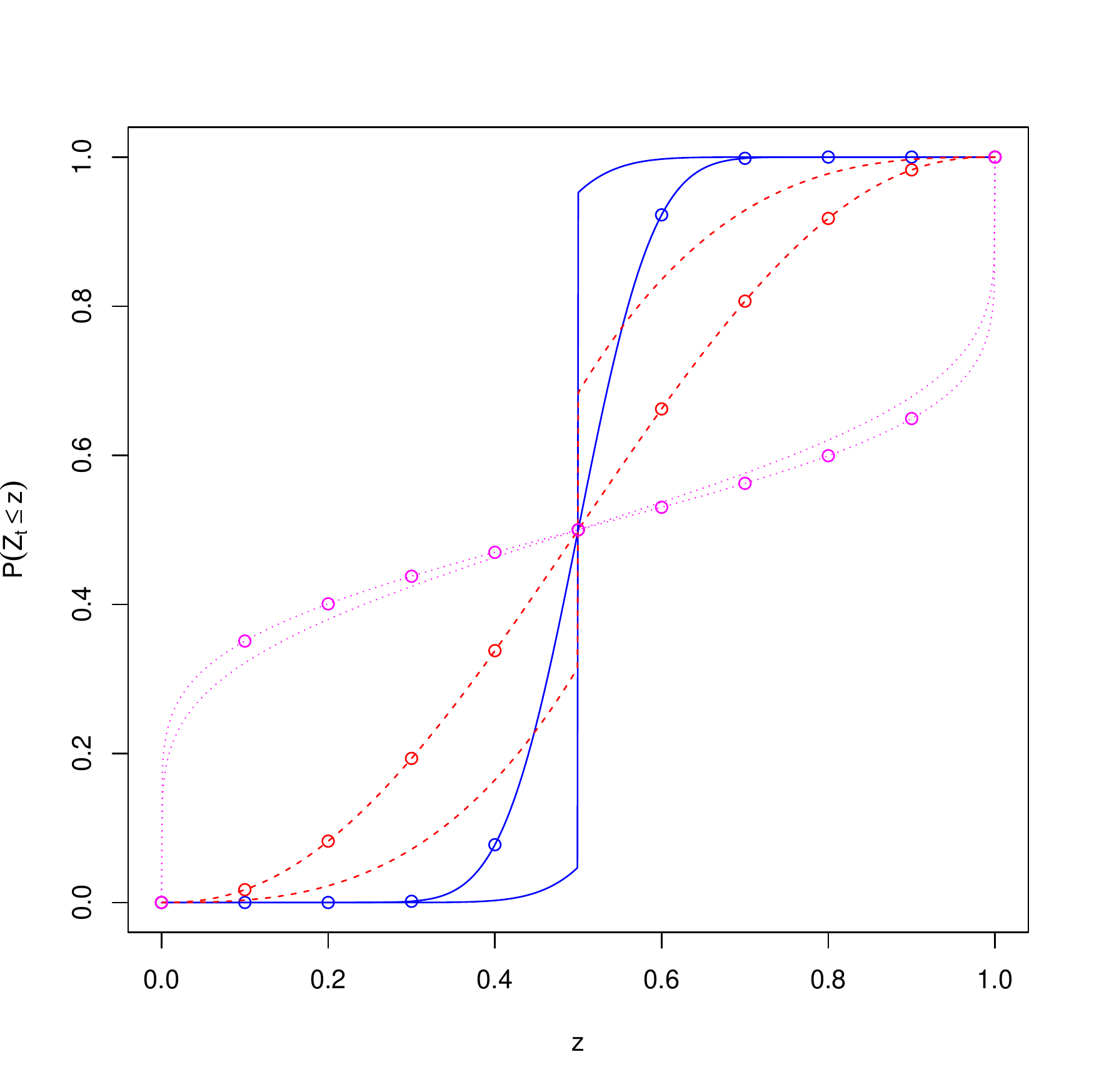}}\hspace{0.2cm}
\subfigure[$Z_0=50\%$ $\eta=15\%$, $\lambda=50\%$]{\includegraphics[width=0.48\columnwidth]{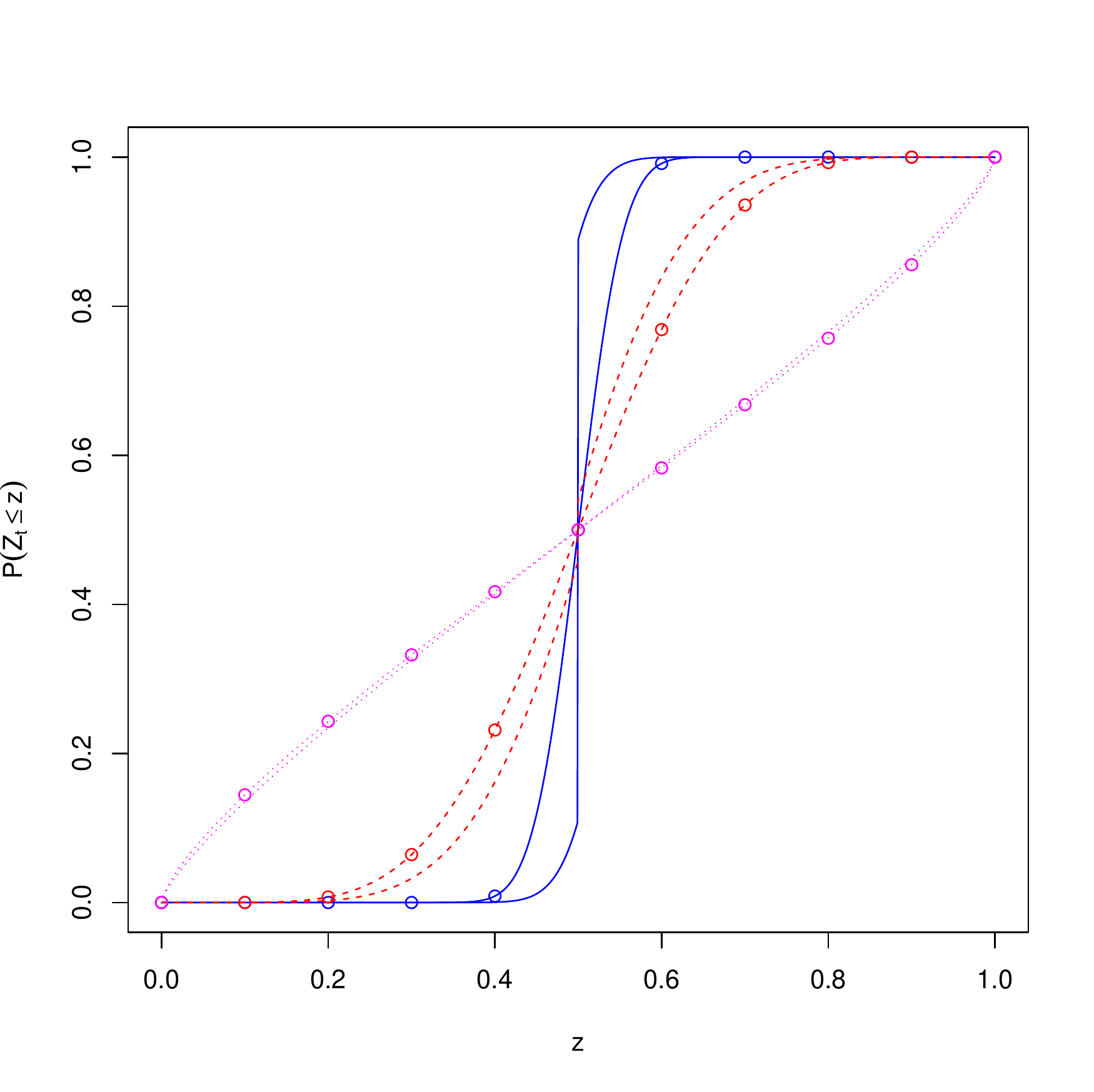}}\\
\subfigure[$Z_0=35\%$ $\eta=15\%$, $\lambda=50\%$]{\includegraphics[width=0.48\columnwidth]{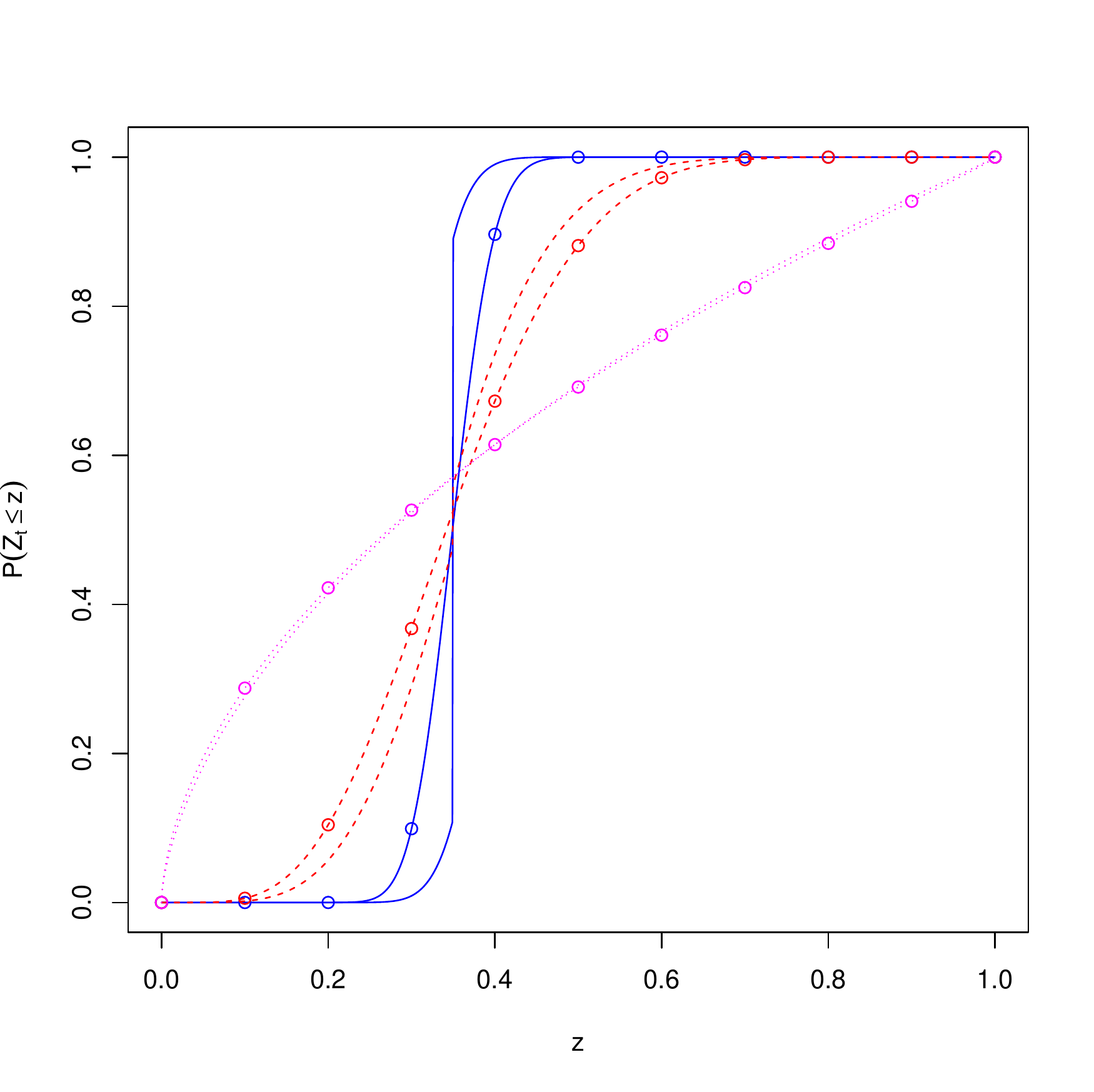}}\hspace{0.2cm}
\subfigure[$Z_0=35\%$, $\eta=25\%$, $\lambda=5\%$]{\includegraphics[width=0.48\columnwidth]{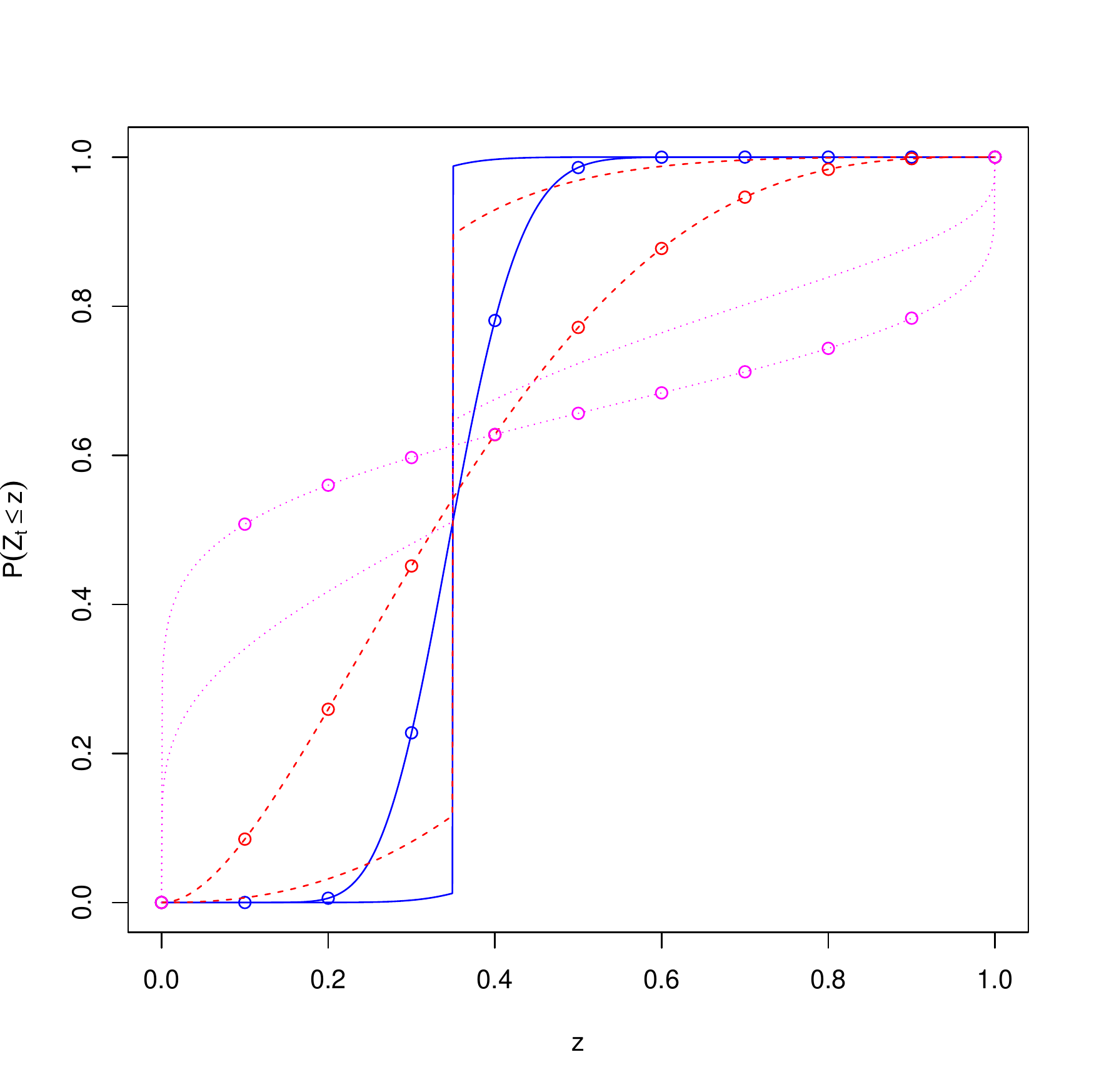}}
\caption{Cumulative distribution function of $\tilde{Z}_t$ (circles) and $Z_t$ (no marker) where $\tilde{Z}$ is the $\Phi$-martingale with initial value $Z_0$ and $t$ equals 0.5 (blue solid), 5 (red, dashed) and 40 (magenta, dotted) years.}\label{fig:CDF}
\end{figure}
%

% ====================================================================================
\section{Conclusion and future research}
% ====================================================================================

In this paper, we focused on the construction of piecewise constant martingales that is, martingales whose trajectories are piecewise constant. Such processes are indeed good candidates to model the dynamics of conditional expectations of random variables under partial (punctual) information. The time-changed approach proves to be quite powerful: starting with a martingale in a given range, we obtain a PWC martingale by using a piecewise constant time-change process. Among those time-change processes that lazy clocks are specifically appealing: these are time-change processes staying always in arrears to the real clock, and that synchronizes to the calendar time at some random times. This ensures that $\theta_t\leq t$ which is a convenient feature when one needs to sample trajectories of the time-change process. Such random times can typically be characterized as last passage times, and enjoy appealing tractability properties. The last jump time of a Poisson process before the current time for instance exhibits a very simple distribution. Other lazy clocks have been proposed as well, based on Brownian motions and Bessel processes, some of which rule out the probability mass at zero. Finally, we provided several martingales time-changed with lazy clocks, called lazy martingales, whose range can be any interval in $\mathbb{R}$ (depending on the range of the latent martingale) and showed that the corresponding distributions can be easily obtained from the law of iterated expectations.\medskip

Yet, tractability and even more importantly, the martingale property result from the independence assumption between the latent martingale and the time-change process. In practice however, it might be more realistic to consider cases where the sample frequency (synchronization rate of the lazy clock $\theta$ to the real clock) depends on the level of the latent martingale $Z$. Finding a tractable model allowing for this coupling remains an open question and is the purpose of future research.

% ====================================================================================
%\section*{Acknowledgments} 
% ====================================================================================
%The authors are grateful to M. Jeanblanc, D. Brigo and K. Yano for stimulating discussions about an earlier version of this manuscript. This research benefited from the support of the ``\textit{Chaire March\'es en Mutation}'', F\'ed\'eration Bancaire Fran\c{c}aise.

% ====================================================================================
% BIBLIOGRAPHY
% ====================================================================================

\ifdefined \MyBib
	\bibliography{\MyBib}
	\bibliographystyle{plain}
\fi

\end{document}